\documentclass[a4paper,reqno,11pt]{amsart}
\usepackage{eepic}
\usepackage{enumerate}
\usepackage{graphicx}
\usepackage{latexsym}
\usepackage{mathrsfs}
\usepackage{tikz}
\usepackage[all]{xy}
\input xy
\xyoption{all}
\usepackage{pstricks}
\usepackage{lscape}
\usepackage{comment}

\numberwithin{equation}{section}
\newtheorem{theorem}{Theorem}[section]

\newtheorem{corollary}[theorem]{Corollary}
\newtheorem{lemma}[theorem]{Lemma}
\newtheorem{proposition}[theorem]{Proposition}

\theoremstyle{definition}
\newtheorem{definition}[theorem]{Definition}
\newtheorem{remark}[theorem]{Remark}
\newtheorem{example}[theorem]{Example}
\newtheorem*{notation}{Notation}
\newcommand{\Mod}{\mathsf{Mod}}

\newcommand{\U}[1]{\underline{#1}}
\newcommand{\add}{\operatorname{\mathsf{add}}\nolimits}

\newcommand{\End}{\operatorname{End}\nolimits}
\newcommand{\Ext}{\operatorname{Ext}\nolimits}

\newcommand{\Hom}{\operatorname{Hom}\nolimits}
\newcommand{\Ho}{{\rm H}}
\renewcommand{\Im}{\operatorname{Im}\nolimits}

\newcommand{\kD}{\mathrm{D}}
\newcommand{\Ker}{\operatorname{Ker}\nolimits}
\newcommand{\Cok}{\operatorname{Cok}\nolimits}

\renewcommand{\mod}{\mathsf{mod}}

\newcommand{\smod}{\operatorname{\underline{\mathsf{mod}}}\nolimits}

\newcommand{\thick}{\operatorname{\mathsf{thick}}\nolimits}

\newcommand{\pd}{\operatorname{proj.dim}\nolimits}
\newcommand{\gl}{\operatorname{gl.dim}\nolimits}

\newcommand{\supp}{\operatorname{Supp}\nolimits}

\newcommand{\proj}{\operatorname{\mathsf{proj}}\nolimits}

\newcommand{\xto}[1]{\xrightarrow{#1}}

\newcommand{\mb}[1]{\mathbb{#1}}
\newcommand{\ms}[1]{\mathsf{#1}}
\newcommand{\mc}[1]{\mathcal{#1}}
\begin{document}
\title[Singularity categories of derived categories of hereditary algebras]{Singularity categories of derived categories of hereditary algebras are derived categories}
\author[Y. Kimura]{Yuta Kimura}
\address{Graduate School of Mathematics, Nagoya University, Frocho, Chikusaku, Nagoya, 464-8602, Japan}
\email{m13025a@math.nagoya-u.ac.jp}
\date{\today}
\begin{abstract}
We show that for the path algebra $A$ of an acyclic quiver, the singularity category of the derived category $\ms{D}^{\rm b}(\mod A)$ is triangle equivalent to the derived category of the functor category of $\smod A$, that is, $\ms{D}_{\rm sg}(\ms{D}^{\rm b}(\mod A))\simeq \ms{D}^{\rm b}(\mod(\smod A))$.
This extends a result in \cite{IO} for the path algebra $A$ of a Dynkin quiver.
An important step is to establish a functor category analog of Happel's triangle equivalence for repetitive algebras.
\end{abstract}
\maketitle
\section{introduction}
Let $k$ be a field and $A$ be a finite dimensional $k$-algebra.
In \cite{IO}, it was shown that if $A$ is a representation finite hereditary algebra, then there exists a triangle equivalence
\begin{align}\label{intro-equ-IO}
\smod\ms{D}^{\rm b}(\mod A)\simeq \ms{D}^{\rm b}(\mod\,B),
\end{align}
where $B$ is the stable Auslander algebra of $A$, $\mod\ms{D}^{\rm b}(\mod A)$ is the Frobenius category of finitely presented functors from $\ms{D}^{\rm b}(\mod A)$ to $\mc{A}b$, and $\smod\ms{D}^{\rm b}(\mod A)$ is its stable category.

In this paper, we extend a triangle equivalence (\ref{intro-equ-IO}) to the case when $A$ is a representation infinite hereditary algebra.
In this case, the role of the stable Auslander algebra is played by the category $\mod(\smod A)$ of finitely presented functors from the stable category $\smod A$ to $\mc{A}b$.
Our main result is the following.

\begin{theorem}[Theorem \ref{thm-hered-stable-derived}]\label{intro-thm-aim}
Let $A$ be a hereditary algebra.
We have a triangle equivalence
\begin{align}\label{intro-equ-main}
\smod\ms{D}^{\rm b}(\mod A)\simeq\ms{D}^{\rm b}(\mod(\smod A)).
\end{align}
\end{theorem}

Note that for a triangulated category $\mc{T}$, the stable category $\smod\mc{T}$ is triangle equivalent to the singularity category $\ms{D}_{\rm sg}(\mc{T})=\ms{D}^{\rm b}(\mod\mc{T})/\ms{K}^{\rm b}(\proj\mc{T})$ \cite{Buchweitz, Orlov} (see Theorem \ref{rickard-theorem}).
Thus (\ref{intro-equ-main}) can be rewritten as $\ms{D}_{\rm sg}(\ms{D}^{\rm b}(\mod A))\simeq\ms{D}^{\rm b}(\mod(\smod A))$.

To prove Theorem \ref{intro-thm-aim}, we need to give general preliminary results on functor categories and repetitive categories.
The functor category $\mod(\smod A)$ is an abelian category with enough projectives and enough injectives since the category $\smod A$ forms a dualizing $k$-variety, which is a distinguished class of $k$-linear categories introduced by Auslander and Reiten \cite{AR74}.
A key role is played by the repetitive category $\ms{R}(\smod A)$ of $\smod A$.
The following our first result implies that $\ms{R}(\smod A)$ is a dualizing $k$-variety.

\begin{theorem}[Theorem \ref{repetitive-dualizing}]\label{into-thm-repetitive-dualizng}
Let $\mc{A}$ be a dualizing $k$-variety.
Then $\ms{R}\mc{A}$ is a dualizing $k$-variety.
\end{theorem}

In particular, $\mod\ms{R}\mc{A}$ is a Frobenius abelian category for any dualizing $k$-variety $\mc{A}$.
We denote by $\smod\ms{R}\mc{A}$ the stable category of $\mod\ms{R}\mc{A}$, which is triangulated.

In the case where $A$ is a representation finite hereditary algebra,
the following Happel's theorem \cite{H} played an important role in the proof of a triangle equivalence (\ref{intro-equ-IO}):
for a finite dimensional $k$-algebra $A$ of finite global dimension, the bounded derived category of $A$ is triangle equivalent to the stable category of the repetitive algebra of $A$.
In Section \ref{section-repetitive}, we show a categorical analog of this triangle equivalence for dualizing $k$-varieties.
In fact, we deal with the following more general class of categories including dualizing $k$-varieties.
For a $k$-linear additive category $\mc{A}$, we denote by $\proj\mc{A}$ the category of finitely generated projective $\mc{A}$-modules and by $\mod\mc{A}$ the category of $\mc{A}$-modules having resolutions by $\proj\mc{A}$.
We consider the following conditions: 
\begin{itemize}
\setlength{\itemsep}{-10pt}
\item[(IFP)]$\kD\mc{A}(X,-)$ is in $\mod\mc{A}$ for each $X\in\mc{A}$, where $\kD=\Hom_{k}(-,k)$.\\
\item[(G)]
$\kD\mc{A}(X,-)$ has finite projective dimension over $\mc{A}$ for each $X\in\mc{A}$.
\end{itemize}
For example, if $\mc{A}$ is a dualizing $k$-variety, then $\mc{A}$ satisfies the condition (IFP).
On the other hand, the condition (G) is a categorical version of Gorensteinness.
Gorenstein-projective modules (also known as Cohen-Macaulay modules, totally reflexive modules) are important class of modules.
We denote by $\ms{GP}(\ms{R}\mc{A},\mc{A})$ the category of Gorenstein-projective $\ms{R}\mc{A}$-modules of finite projective dimension as $\mc{A}$-modules.
We prove the following.

\begin{theorem}[Corollaries \ref{cor-happel-thm}, \ref{happel-thm-dualizing}]\label{intro-thm-happel}
Let $\mc{A}$ be a $k$-linear, Hom-finite additive category.
\begin{itemize}
\item[{\rm (a)}]
Assume that  $\mc{A}$ and $\mc{A}^{\rm op}$ satisfy (IFP) and (G).
Then we have a triangle equivalence \[\ms{K}^{\rm b}(\proj\mc{A})\simeq\U{\ms{GP}}(\ms{R}\mc{A},\mc{A}).\]
\item[{\rm (b)}]
Assume that $\mc{A}$ is a dualizing $k$-variety.
If each object of $\mod\mc{A}$ and $\mod\mc{A}^{\rm op}$ has finite projective dimension,
then we have a triangle equivalence
\begin{align*}
\ms{D}^{\rm b}(\mod\mc{A})\simeq\smod\ms{R}\mc{A}.
\end{align*}
\end{itemize}
\end{theorem}

We refer to \cite{BGG, IO, Kimura1, Kimura2, Lu, Minamoto-Yamaura, Mori-Ueyama, Y} for recent results which realize stable categories as derived categories in different settings.

In Section \ref{section-application-hereditary}, we show the following theorem, which together with Theorem \ref{intro-thm-happel} implies Theorem \ref{intro-thm-aim}.

\begin{theorem}[Theorem \ref{thm-rep-derived}]\label{intro-thm-rep-derived}
Let $A$ be a representation infinite hereditary algebra.
Then we have an equivalence of additive categories 
\begin{align*}
\ms{R}(\smod A) \simeq \ms{D}^{\rm b}(\mod A).
\end{align*}
\end{theorem}

\begin{notation}
In this paper, we denote by $k$ a field.
All subcategories are full and closed under isomorphisms.
Let $\mc{C}$ be an additive category and $\mc{S}$ be a subclass of objects of $\mc{C}$ or a subcategory of $\mc{C}$.
We denote by $\add\mc{S}$ the subcategory of $\mc{C}$ whose objects are direct summands of finite direct sums of objects in $\mc{S}$.
For subcategories $\mc{C}_{i}$ ($i\in I$) of $\mc{C}$, we denote by $\bigvee_{i\in I}\mc{C}_{i}$ the smallest additive subcategory of $\mc{C}$ containing all $\mc{C}_{i}$ and closed under direct summands.
For objects $X,Y\in\mc{C}$, we denote by $\mc{C}(X,Y)$ the set of morphisms from $X$ to $Y$ in $\mc{C}$.
We call a category \emph{skeletally small} if the class of isomorphism class of objects is a set.
We assume that all categories in this paper are skeletally small.
\end{notation}
\section{Preliminaries}
\subsection{Functor categories}
In this subsection, we recall the definition of modules over categories.
Let $\mc{A}$ be an additive category.
An \emph{$\mc{A}$-module} is a contravariant additive functor from $\mc{A}$ to $\mc{A}b$, where $\mc{A}b$ is the category of abelian groups.
We denote by $\Mod\mc{A}$ the category of $\mc{A}$-modules, where morphisms of $\Mod\mc{A}$ are morphisms of functors.
Since $\mc{A}$ is skeletally small, $\Mod\mc{A}$ is a category.
It is well known that $\Mod\mc{A}$ is abelian.

For two morphisms $f : L\to M$ and $g : M \to N$ of $\Mod\mc{A}$, the sequence $L \to M \to N$ is exact in $\Mod\mc{A}$ if and only if the induced sequence $L(X) \to M(X) \to N(X)$ is exact in $\mc{A}b$ for any $X\in\mc{A}$.
\begin{example}
For each $X\in\mc{A}$, we have an $\mc{A}$-module $\mc{A}(-,X)$.
By Yoneda's lemma, $\mc{A}(-,X)$ is projective in $\Mod\mc{A}$.
\end{example}
The following notation is basic and used throughout this paper.
We call an $\mc{A}$-module $M$ \emph{finitely generated} if there exists an epimorphism $\mc{A}(-,X)\to M$ in $\Mod\mc{A}$ for some $X\in\mc{A}$.
We denote by $\proj\mc{A}$ the subcategory of $\Mod\mc{A}$ consisting of all finitely generated projective $\mc{A}$-modules.
Note that finitely generated projective modules are precisely direct summands of representable functors.
We need the following notation which is called $FP_{n}$ in some literatures (e.g. \cite{SGA6, Brown}).
\begin{definition}\label{def-FP-n}
Let $\mc{A}$ be an additive category and $n\geq 0$ be an integer.
\begin{itemize}
\item[(1)]
We denote by $\mod_{n}\mc{A}$ the subcategory of $\Mod\mc{A}$ consisting of all $\mc{A}$-modules $M$ such that  there exists an exact sequence 
\begin{align*}
P_{n} \to \cdots \to P_{1} \to P_{0}\to M\to0
\end{align*} 
in $\Mod\mc{A}$, where $P_{i}$ is in $\proj\mc{A}$ for each $0\leq i \leq n$.
\item[(2)]
We denote by $\mod\mc{A}$ the subcategory of $\Mod\mc{A}$ consisting of all $\mc{A}$-modules $M$ such that there exists an exact sequence 
\begin{align*}
\cdots \to P_{2} \to P_{1}\to P_{0}\to M\to0
\end{align*} 
in $\Mod\mc{A}$, where $P_{i}$ is in $\proj\mc{A}$ for each $i\geq 0$.
\end{itemize}
\end{definition}
The following lemma is a basic observation on $\mod_{n}\mc{A}$.
\begin{lemma}\label{lem-Schanuel}
The following statements hold for an additive category $\mc{A}$.
\begin{itemize}
\item[{\rm (a)}]
Let $M\in\mod_{n}\mc{A}$.
Assume that there exists an exact sequence $P_{l}\to P_{l-1}\to\cdots \to P_{0} \to M \to 0$ with $P_{i}\in\proj\mc{A}$ and $l\leq n$.
Then there exist $P_{l+1},\ldots,P_{n}\in\proj\mc{A}$ and an exact sequence $P_{n}\to P_{n-1}\to\cdots \to P_{0} \to M \to 0$.
\item[{\rm (b)}]
Let $M\in\Mod\mc{A}$.
Assume that there exist the following two exact sequences
\begin{align*}
0 \to K \to P_{n} \to P_{n-1} \to \cdots \to P_{0} \to M \to 0,\\
0 \to L \to Q_{n} \to Q_{n-1} \to \cdots \to Q_{0} \to M \to 0,
\end{align*}
where $P_{i}, Q_{i}\in\proj\mc{A}$ for each $i\geq 0$.
Then there exist $P,Q\in\proj\mc{A}$ such that $K\oplus P\simeq L\oplus Q$.
\end{itemize}
\end{lemma}
\begin{proof}
(a)
This follows from (b).

(b)
The case where $n=0$ is well known as Schanuel's Lemma.
The case where $n>0$ is shown by an induction on $n$ and by using the case where $n=0$.
\end{proof}
The following lemma gives a sufficient condition when an $\mc{A}$-module is in $\mod_{n}\mc{A}$.
For simplicity, we use the notation $\mod_{-1}\mc{A}:=\Mod\mc{A}$, $\mod_{\infty}\mc{A}:=\mod\mc{A}$ and $\infty-1:=\infty$.
\begin{lemma}\label{M-in-mod-n-A}
Let $\mc{A}$ be an additive category and $M$ be an $\mc{A}$-module.
Then we have the following properties.
\begin{itemize}
\item[{\rm (a)}]
Let $n\geq0$ be an integer.
If there exists an exact sequence
$X_{n}\to X_{n-1} \to \cdots \to X_{0} \to M \to 0$
in $\Mod\mc{A}$ with $X_{i}\in\mod_{n-i}\mc{A}$ for any $0\leq i \leq n$,
then we have $M\in\mod_{n}\mc{A}$.
\item[{\rm (b)}]
If there exists an exact sequence
$\cdots \to X_{2} \to X_{1} \to X_{0} \to M \to 0$
in $\Mod\mc{A}$ with $X_{i}\in\mod\mc{A}$ for any $i\geq 0$,
then we have $M\in\mod\mc{A}$.
\item[{\rm (c)}]
Let $n\in\mb{Z}_{\geq 0}\cup\{\infty\}$.
For an exact sequence $0 \to L \to M \to N\to 0$ in $\Mod\mc{A}$ with $L\in\mod_{n-1}\mc{A}$ and $M\in\mod_{n}\mc{A}$, we have $N\in\mod_{n}\mc{A}$.
\end{itemize}
\end{lemma}
\begin{proof}
(a)
We have the following commutative diagram
$$
\xymatrix@=12pt{
X_{n} \ar[r]& X_{n-1} \ar[r]& \cdots \ar[r]& X_{0} \ar[r]& M \ar[r]& 0 \\
P_{n, 0} \ar[r]\ar[u]& P_{n-1, 0} \ar[r]\ar[u]& \cdots \ar[r]& P_{0,0} \ar[u]& & \\
	& P_{n-1,1} \ar[r]\ar[u]& \cdots \ar[r]& P_{0,1} \ar[u]& & \\
	&	&	& \cdots \ar[u]& & \\
	&	&	& P_{0,n} \ar[u]& &\\
}
$$
in $\Mod\mc{A}$, where each $P_{i,0} \to X_{i}$ is epimorphism for $0\leq i \leq n$, each vertical sequence is exact and each $P_{i,j}$ is in $\proj\mc{A}$.
Thus we have an exact sequence 
\begin{align*}
\overline{P}_{n} \to \cdots \to \overline{P}_{1}\to \overline{P}_{0}\to M\to0
\end{align*} 
in $\Mod\mc{A}$, where $\overline{P}_{i}=\bigoplus_{j=0}^{i}P_{j,i-j}$ for $0\leq i \leq n$.
Since $\overline{P}_{i}$ is in $\proj\mc{A}$ for $0\leq i \leq n$, $M$ is an object of $\mod_{n}\mc{A}$.

(b) This comes from the same argument as (a).

(c) This follows from (a) for $n\in\mb{Z}_{\geq 0}$ and (b) for $n=\infty$.
\end{proof}
Let $\mc{A}$ be an abelian category and $\mc{B}$ be a subcategory of $\mc{A}$.
We say that $\mc{B}$ is a \emph{thick} subcategory of $\mc{A}$ if $\mc{B}$ is closed under direct summands and for any exact sequence $0\to X \to Y \to Z \to 0$ in $\mc{A}$, if two of $X,Y,Z$ are in $\mc{A}$, then so is the third.
We have the following observation of the categories $\mod_{n}\mc{A}$.
\begin{lemma}\label{cap-modnA-modinfA}
Let $\mc{A}$ be an additive category.
Then we have the following statements.
\begin{itemize}
\item[{\rm (a)}]
$\mod_{n}\mc{A}$ is closed under extensions and direct summands in $\Mod\mc{A}$ for each $n\geq 0$.
\item[{\rm (b)}]
$\mod\mc{A}=\bigcap_{n\geq 0}\mod_{n}\mc{A}$ holds.
\item[{\rm (c)}]
{\rm (e.g. \cite[Proposition 2.6]{Enomoto})} $\mod\mc{A}$ is a thick subcategory of $\Mod\mc{A}$.
\end{itemize}
\end{lemma}
\begin{proof}
(a)
By Horseshoe Lemma, $\mod_{n}\mc{A}$ is closed under extensions in $\Mod\mc{A}$.
Let $X\oplus Y\in\mod_{n}\mc{A}$.
We show that $X,Y\in\mod_{n}\mc{A}$ by an induction on $n$.
If $n=0$, then the claim is clear.
Assume $n>0$.
Since $X\oplus Y\in\mod_{n}\mc{A}\subset\mod_{n-1}\mc{A}$ holds,
by the inductive hypothesis, we have $X,Y\in\mod_{n-1}\mc{A}$.
Then by Lemma \ref{M-in-mod-n-A} (c), we have $X,Y\in\mod_{n}\mc{A}$.

(b)
In general $\mod\mc{A}\subset\mod_{n}\mc{A}$ holds for each $n\geq 0$.
The converse follows from Lemma \ref{lem-Schanuel} (a).

(c)
By (a) and (b), $\mod\mc{A}$ is closed under extensions and direct summands.
Let $0\to L \to M \to N \to 0$ be an exact sequence in $\Mod\mc{A}$.
By Lemma \ref{M-in-mod-n-A} (c), if $L, M \in \mod\mc{A}$, then $N\in\mod\mc{A}$ holds.
Assume that $M,N\in\mod\mc{A}$.
There exists an exact sequence $0 \to \Omega N \to P \to N \to 0$ such that $P\in\proj\mc{A}$ and $\Omega N\in\mod\mc{A}$.
By taking a pull-back diagram of $M\to N \leftarrow P$, we have an exact sequence $0 \to \Omega N \to P\oplus L \to M \to 0$.
Since $\mod\mc{A}$ is closed under extensions and direct summands, we have $L\in\mod\mc{A}$.
\end{proof}
\subsection{Gorenstein-projective modules}
We define Gorenstein-projective modules.
Let $\mc{A}$ be an additive category.
We first define a contravariant functor
\begin{align*}
(-)^{\ast} : \Mod\mc{A}\to\Mod\mc{A}^{\rm op}
\end{align*}
as follows: for $M\in\Mod\mc{A}$ and $X\in\mc{A}$, let $(M)^{\ast}(X):=(\Mod\mc{A})(M,\mc{A}(-,X))$.
By the same way, we define a contravariant functor $(-)^{\ast} : \Mod\mc{A}^{\rm op}\to\Mod\mc{A}$.
Let $P_{\bullet}:=(P_{i}, d_{i}: P_{i}\to P_{i+1})_{i\in\mb{Z}}$ be a complex of finitely generated projective $\mc{A}$-modules.
We say that $P_{\bullet}$ is \emph{totally acyclic} if complexes $P_{\bullet}$ and $\cdots \to (P_{i+1})^{\ast} \to (P_{i})^{\ast} \to (P_{i-1})^{\ast} \to \cdots$ are acyclic.
\begin{definition}\label{def-goren-proj}
Let $\mc{A}$ be an additive category.
An $\mc{A}$-module $M$ is said to be \emph{Gorenstein-projective} if there exists a totally acyclic complex $P_{\bullet}$ such that $\Im d_{0}$ is isomorphic to $M$.
We denote by $\ms{GP}\mc{A}$ the full subcategory of $\Mod\mc{A}$ consisting of all Gorenstein-projective $\mc{A}$-modules.
\end{definition}
For instance, a finitely generated projective $\mc{A}$-module is Gorenstein-projective.
In general, $\ms{GP}\mc{A}\subset\mod\mc{A}$ holds.
We see a fundamental properties of Gorenstein-projective modules.

Let $\mc{W}$ be a subcategory of $\Mod\mc{A}$.
We denote by ${}^{\perp}\mc{W}$ the subcategory of $\Mod\mc{A}$ consisting of $\mc{A}$-modules $M$ satisfying $\Ext_{\Mod\mc{A}}^{i}(M,W)=0$ for any $W\in\mc{W}$ and any $i>0$.
We denote by $\mc{X}_{\mc{W}}$ the subcategory of ${}^{\perp}\mc{W}$ consisting of $\mc{A}$-modules $M$ such that there exists an exact sequence $0\to M \to W_{0} \xto{f_{0}} W_{1} \xto{f_{1}} \cdots$ with $W_{i}\in\mc{W}$ and $\Im f_{i}\in{}^{\perp}\mc{W}$ for any $i\geq 0$.
By \cite[Proposition 5.1]{AR91}, $\mc{X}_{\proj\mc{A}}$ is closed under extensions, direct summands and kernels of epimorphisms in $\Mod\mc{A}$.
\begin{lemma}\label{lem-GP-dual-thick}
Let $\mc{A}$ be an additive category.
Then the following holds.
\begin{itemize}
\item[{\rm (a)}]
The functor $(-)^{\ast} : \Mod\mc{A}\to\Mod\mc{A}^{\rm op}$ induces a duality $(-)^{\ast} : \ms{GP}\mc{A}\to\ms{GP}\mc{A}^{\rm op}$.
\item[{\rm (b)}]
$\mc{X}_{\proj\mc{A}}\cap\mod\mc{A}=\ms{GP}\mc{A}$ holds.
In particular, $\ms{GP}\mc{A}$ is closed under extensions, direct summands and kernels of epimorphisms in $\Mod\mc{A}$.
\end{itemize}
\end{lemma}
\begin{proof}
(a)
This follows from the definition of $\ms{GP}\mc{A}$ and the fact that $(-)^{\ast}$ induces a duality between $\proj\mc{A}$ and $\proj\mc{A}^{\rm op}$.

(b)
In general $\mc{X}_{\proj\mc{A}}\cap\mod\mc{A}\supset\ms{GP}\mc{A}$ holds.
If $M\in\mc{X}_{\proj\mc{A}}\cap\mod\mc{A}$, then there exists an exact sequence $P_{\bullet}=(P_{i}, d_{i}: P_{i}\to P_{i+1})_{i\in\mb{Z}}$,
where $M\simeq \Im d_{0}$, $P_{i}\in\proj\mc{A}$ for any $i\in\mb{Z}$ and $\Im d_{i}\in{}^{\perp}(\proj\mc{A})$ for any $i\geq 1$.
Then this sequence is totally acyclic, since $\Im d_{i}\in{}^{\perp}(\proj\mc{A})$ holds for any $i\geq 1$.
\end{proof}
Let $\mc{B}$ be an extension closed subcategory of an abelian category $\mc{A}$.
An exact sequence in $\mc{A}$ is called an exact sequence in $\mc{B}$ if each term of it is an object of $\mc{B}$.
We say that an object $Z$ in $\mc{B}$ is \emph{relative-projective} if any exact sequence $0\to X\to Y \to Z \to 0$ in $\mc{B}$ splits.
Dually, we define \emph{relative-injective} objects.
We say that $\mc{B}$ has \emph{enough projectives} if for any $X\in\mc{B}$, there exists an exact sequence $0\to Z \to P \to X \to 0$ in $\mc{B}$ such that $P$ is relative-projective.
Dually, we define a subcategory of $\mc{A}$ which has \emph{enough injectives}.
An extension closed subcategory $\mc{B}$ of $\mc{A}$ is said to be \emph{Frobenius} if $\mc{B}$ has enough projectives, enough injectives and the relative-projective objects coincide with the relative-injective objects.

The following observation is immediate (cf. \cite{XWChen}).
\begin{proposition}\label{GP-Frobenius}
Let $\mc{A}$ be an additive category.
Then $\ms{GP}\mc{A}$ is a Frobenius category, where the relative-projective objects are precisely finitely generated $\mc{A}$-modules.
\end{proposition}
\begin{proof}
$\ms{GP}\mc{A}$ is extension closed in $\Mod\mc{A}$ by Lemma \ref{lem-GP-dual-thick} (b).
By the definition of $\ms{GP}\mc{A}$ and the duality $(-)^{\ast} : \ms{GP}\mc{A} \to \ms{GP}\mc{A}^{\rm op}$, $\ms{GP}\mc{A}$ has enough projectives and enough injectives.
Again by the definition of $\ms{GP}\mc{A}$, the relative-projective objects coincide with the relative-injective objects, which coincide with finitely generated projective $\mc{A}$-modules.
\end{proof}
\subsection{Dualizing $k$-varieties and Serre dualities}
In this subsection, we recall the definition of dualizing $k$-varieties.
Let $\mc{A}$ be an additive category.
We call an object of $\mod_{1}\mc{A}$ a \emph{finitely presented} $\mc{A}$-module.

A morphism $X\to Y$ in $\mc{A}$ is a \emph{weak kernel} of a morphism $Y \to Z$ if the induced sequence $\mc{A}(-,X)\to\mc{A}(-,Y)\to\mc{A}(-,Z)$ is exact in $\Mod\mc{A}$.
We say that $\mc{A}$ has weak kernels if each morphism in $\mc{A}$ has a weak kernel.
The following lemma says when an additive category has weak kernels.
\begin{lemma}\label{lem-weak-mod-modinf}
Let $\mc{A}$ be an additive category.
The following statements are equivalent.
\begin{itemize}
\item[{\rm (i)}]
$\mc{A}$ has weak kernels.
\item[{\rm (ii)}]
$\mod_{1}\mc{A}$ is abelian.
\item[{\rm (iii)}]
$\mod_{1}\mc{A}=\mod\mc{A}$ holds.
\end{itemize}
\end{lemma}
\begin{proof}
It is well known that the statements (i) and (ii) are equivalent.
The statements (i) and (iii) are equivalent by \cite[Proposition 2.7]{Enomoto}. 
\end{proof}
Let $\mc{A}$ be an additive category and $X\in\mc{A}$.
A morphism $e : X\to X$ in $\mc{A}$ is called an \emph{idempotent} if $e^{2}=e$.
We call $\mc{A}$ \emph{idempotent complete} if each idempotent of $\mc{A}$ has a kernel.

Let $k$ be a field.
A \emph{$k$-linear category} $\mc{A}$ is a category such that $\mc{A}(X,Y)$ admits a structure of $k$-modules and the composition of morphisms of $\mc{A}$ is $k$-bilinear.
A contravariant functor $F : \mc{A} \to \mc{B}$ between $k$-linear categories are called \emph{$k$-functor} if $F_{X,Y} : \mc{A}(X,Y)\to\mc{B}(FY,FX)$ is $k$-linear for any $X,Y\in\mc{A}$. 
If $\mc{A}$ is an additive $k$-linear category, then any $\mc{A}$-module can be regarded as a contravariant additive $k$-functor from $\mc{A}$ to $\Mod k$, where $\Mod k$ is the category of $k$-modules.

Let $\mc{A}$ be a $k$-linear additive category.
We call $\mc{A}$ \emph{Hom-finite} if $\mc{A}(X,Y)$ is finitely generated over $k$ for any $X,Y\in\mc{A}$.
We recall one proposition about the Krull-Schmidt property of $k$-linear additive categories.
\begin{proposition}\label{prop-Krull-Schmidt}
Let $\mc{A}$ be a $k$-linear, Hom-finite additive category.
Then the following properties are equivalent.
\begin{itemize}
\item[{\rm (i)}]
$\mc{A}$ is idempotent complete.
\item[{\rm (ii)}]
The endomorphism algebra of each indecomposable object in $\mc{A}$ is local.
\item[{\rm (iii)}]
$\mc{A}$ is Krull-Schmidt, that is, each object of $\mc{A}$ is a finite direct sum of objects whose endomorphism algebras are local.
\end{itemize}
Moreover the decomposition of {\rm (iii)} is unique up to isomorphism.
\end{proposition}
\begin{proposition}\label{ex-Krull-Schmidt}
Let $\mc{A}$ be a $k$-linear, Hom-finite additive category.
Then $\mod\mc{A}$ is Krull-Schmidt.
In particular, each object of $\mod\mc{A}$ has a minimal projective resolution.
\end{proposition}
\begin{proof}
Since $\mod\mc{A}$ is closed under direct summands in $\Mod\mc{A}$, $\mod\mc{A}$ is idempotent complete.
$\mod\mc{A}$ is Hom-finite, since $\mc{A}$ is Hom-finite.
\end{proof}
We recall the definition of dualizing $k$-varieties.
Let $\mc{A}$ be a $k$-linear additive category.
We have contravariant functors $\kD : \Mod\mc{A}\to\Mod\mc{A}^{\rm op}$ and $\kD : \Mod\mc{A}^{\rm op}\to\Mod\mc{A}$ given by $(\kD M)(X):=\kD(M(X))$.
\begin{definition}
Let $\mc{A}$ be a $k$-linear, Hom-finite, idempotent complete additive category.
We call $\mc{A}$ a \emph{dualizing $k$-variety} if the functor $\kD : \Mod\mc{A}\to\Mod\mc{A}^{\rm op}$ induces a duality between $\mod_{1}\mc{A}$ and $\mod_{1}\mc{A}^{\rm op}$.
\end{definition}
The following is typical examples of dualizing $k$-varieties.
\begin{example}\cite{AR74}\label{example-dualizing}
\begin{itemize}
\item[{\rm (a)}]
If $\mc{A}$ is a dualizing $k$-variety, then $\mc{A}^{\rm op}$ is a dualizing $k$-variety.
\item[{\rm (b)}]
Let $A$ be a finite dimensional $k$-algebra and $\mod A$ be the category of finitely generated $A$-modules. Let $\proj A$ be the full subcategory of $\mod A$ consisting of all finitely generated projective $A$-modules.
Then $\mod A$ and $\proj A$ are dualizing $k$-varieties.
\end{itemize}
\end{example}
We  state some properties of dualizing $k$-varieties.
\begin{lemma}\cite{AR74}\label{lem-properties-dualizing}
Let $\mc{A}$ be a dualizing $k$-variety, then we have the following properties.
\begin{itemize}
\item[{\rm (a)}]
$\mc{A}$ and $\mc{A}^{\rm op}$ have weak kernels.
\item[{\rm (b)}]
$\mod\mc{A}$ is a dualizing $k$-variety.
\item[{\rm (c)}]
Each object in $\mod\mc{A}$ has a projective cover and an injective hull.
\end{itemize}
\end{lemma}
Let $\mc{A}$ be a $k$-linear, $\Hom$-finite additive category.
A \emph{Serre functor} on $\mc{A}$ is an auto-equivalence $\mb{S} : \mc{A} \to \mc{A}$ such that there exists a bifunctorial isomorphism 
\begin{align*}
\Hom_{\mc{A}}(X,Y)\simeq\kD\Hom_{\mc{A}}(Y,\mb{S}(X))
\end{align*}
for any $X,Y\in\mc{A}$.
We denote by $\mb{S}^{-1}$ a quasi-inverse of $\mb{S}$.
It is easy to see that if $\mc{A}$ has a Serre functor $\mb{S}$, then $\mc{A}^{\rm op}$ has a Serre functor $\mb{S}^{-1}$.

If $\mc{A}$ has a Serre functor $\mb{S}$, then $(-)^{\ast}$ is described as in the following lemma.
Since $\mb{S}$ is an auto-equivalence, we have an equivalence $\Mod\mc{A}\to\Mod\mc{A}$ given by $M \mapsto M\circ\mb{S}^{-1}$.
By composing the functor $\kD : \Mod\mc{A} \to \Mod\mc{A}^{\rm op}$, we have a contravariant functor $\Mod\mc{A} \to \Mod\mc{A}^{\rm op}$ given by $M \mapsto \kD(M\circ\mb{S}^{-1})$.
We denote by $\Mod_{\rm fg}\mc{A}$ the subcategory of $\Mod\mc{A}$ consisting of $\mc{A}$-modules $M$ such that $M(X)$ is finitely generated over $k$ for any $X\in\mc{A}$.
Note that $\kD$ induces a duality $\Mod_{\rm fg}\mc{A}\to\Mod_{\rm fg}\mc{A}^{\rm op}$ and the categories $\mod_{0}\mc{A}$ and $\ms{GP}\mc{A}$ are contained in $\Mod_{\rm fg}\mc{A}$.
\begin{lemma}\label{lem-Serre-Adual}
Let $\mc{A}$ be a $k$-linear, Hom-finite additive category with a Serre functor $\mb{S}$.
Then the following statements hold.
\begin{itemize}
\item[{\rm (a)}]
We have an isomorphism of functors $(-)^{\ast}\simeq \kD(-\circ\mb{S}^{-1}) : \Mod_{\rm fg}\mc{A} \to \Mod_{\rm fg}\mc{A}^{\rm op}$, and this functor is a duality.
\item[{\rm (b)}]
Let $M\in\Mod\mc{A}$.
The following statements are equivalent.
\begin{itemize}
\item[{\rm (i)}]
$M\in\ms{GP}\mc{A}$.
\item[{\rm (ii)}]
$M\in\mod\mc{A}$ and $M^{\ast}\in\mod\mc{A}^{\rm op}$.
\end{itemize}
\end{itemize}
\end{lemma}
\begin{proof}
(a)
Let $M\in\Mod_{\rm fg}\mc{A}$ and $X\in\mc{A}$.
We have the following equalities.
\begin{align*}
(M)^{\ast}(X) & =(\Mod\mc{A})(M,\mc{A}(-,X))\\
	& \simeq (\Mod\mc{A}^{\rm op})(\kD\mc{A}(-,X),\kD M)\\
	& \simeq (\Mod\mc{A}^{\rm op})(\mc{A}(\mb{S}^{-1}(X),-),\kD M)\\
	& \simeq \kD(M(\mb{S}^{-1}(X))),
\end{align*}
which functorial on $X$.
Thus we have an isomorphism of functors $(-)^{\ast}\simeq \kD(-\circ\mb{S}^{-1})$.
This functor is a duality, since $\kD$ is a duality and $\mb{S}$ is an equivalence.

(b)
Assume that $M\in\ms{GP}\mc{A}$.
By Lemma \ref{lem-GP-dual-thick} (a), we have $M^{\ast}\in\ms{GP}\mc{A}^{\rm op}$.
In general $\ms{GP}\mc{A}\subset\mod\mc{A}$ holds, thus (i) implies (ii).
Assume that (ii) holds.
There exists an exact sequence $\cdots \to Q_{2} \to Q_{1} \to M^{\ast} \to 0$, where $Q_{i}\in\proj\mc{A}^{\rm op}$.
By (a), $(-)^{\ast}$ is an exact functor.
Therefore we have an exact sequence \[\cdots \to P_{2} \to P_{1} \to P_{0} \xto{d} Q_{1}^{\ast} \to Q_{2}^{\ast} \to \cdots,\]
where $P_{i}, Q_{i}^{\ast}\in\proj\mc{A}$ and $\Im d\simeq M$.
This exact sequence is totally acyclic, since $(-)^{\ast}$ is exact.
We have $M\in\ms{GP}\mc{A}$.
\end{proof}
Later we use the following characterization of dualizing $k$-varieties with Serre functors.
\begin{proposition}\label{dualizing-Serre-Frobenius}
Let $\mc{A}$ be a $k$-linear, Hom-finite, idempotent complete additive category.
Then the following statements are equivalent.
\begin{itemize}
\item[{\rm (i)}]
$\mc{A}$ is a dualizing $k$-variety and has a Serre functor.
\item[{\rm (ii)}]
$\mc{A}$ and $\mc{A}^{\rm op}$ have weak kernels and $\mc{A}$ has a Serre functor.
\item[{\rm (iii)}]
$\ms{GP}\mc{A}=\mod_{1}\mc{A}$, $\ms{GP}\mc{A}^{\rm op}=\mod_{1}\mc{A}^{\rm op}$ hold and $\kD\mc{A}(X,-)\in\mod_{1}\mc{A}$, $\kD\mc{A}(-,X)\in\mod_{1}\mc{A}^{\rm op}$ hold for any $X\in\mc{A}$.
\end{itemize}
\end{proposition}
\begin{proof}
By Lemma \ref{lem-properties-dualizing}, {\rm (i)} implies {\rm (ii)}.
We show that {\rm (ii)} implies {\rm (i)}.
Let $M\in\mod_{1}\mc{A}$.
We show that $\kD M$ is in $\mod_{1}\mc{A}^{\rm op}$.
There exists an exact sequence $P_{1}\to P_{0}\to M \to 0$ for some $P_{1},P_{0}\in\proj\mc{A}$.
By the functor $\kD : \Mod\mc{A}\to\Mod\mc{A}^{\rm op}$, we have an exact sequence $0\to \kD M \to \kD P_{0}\to\kD P_{1}$ in $\Mod\mc{A}$.
Since $\mc{A}$ has a Serre functor, we have $\kD P_{1}, \kD P_{0}\in\proj\mc{A}^{\rm op}$.
Since $\mc{A}^{\rm op}$ has weak kernels, $\kD M$ is in $\mod_{1}\mc{A}^{\rm op}$.
By the dual argument, for any $N\in\mod_{1}\mc{A}^{\rm op}$, we have $\kD N\in\mod_{1}\mc{A}$.
Thus $\kD : \mod_{1}\mc{A}\to\mod_{1}\mc{A}^{\rm op}$ is a duality.

We show that {\rm (i)} implies {\rm (iii)}.
Since $\mc{A}$ is a dualizing $k$-variety, $\kD\mc{A}(X,-)\in\mod_{1}\mc{A}$, $\kD\mc{A}(-,X)\in\mod_{1}\mc{A}^{\rm op}$ hold for any $X\in\mc{A}$.
By Lemma \ref{lem-weak-mod-modinf}, we have $\mod\mc{A}=\mod_{1}\mc{A}$ and $\mod\mc{A}^{\rm op}=\mod_{1}\mc{A}^{\rm op}$.
In general $\ms{GP}\mc{A}\subset\mod\mc{A}$ holds.
Let $M\in\mod\mc{A}$.
We show that $M\in\ms{GP}\mc{A}$.
Since $\mc{A}$ is a dualizing $k$-variety, $\kD M\in\mod\mc{A}^{\rm op}$ holds.
By Lemma \ref{lem-Serre-Adual} (a), $M^{\ast}\in\mod\mc{A}^{\rm op}$ holds.
Thus by Lemma \ref{lem-Serre-Adual} (b), $M\in\ms{GP}\mc{A}$ holds.

We show that {\rm (iii)} implies {\rm (ii)}.
In general, $\ms{GP}\mc{A}\subset \mod\mc{A}\subset \mod_{1}\mc{A}$ holds.
Therefore by Lemma \ref{lem-weak-mod-modinf}, $\mc{A}$ and $\mc{A}^{\rm op}$ have weak kernels.
Consider the functor $\kD\circ(-)^{\ast} : \Mod\mc{A}\to\Mod\mc{A}$.
This functor induces an equivalence $\proj\mc{A} \xto{\sim}\proj\mc{A}$.
In fact, if $M\in\proj\mc{A}$, then $M^{\ast}\in\proj\mc{A}^{\rm op}$.
By the assumption, we have $\kD(M^{\ast})\in\mod_{1}\mc{A}=\ms{GP}\mc{A}$.
Since $\kD : \Mod_{\rm fg}\mc{A}^{\rm op} \to \Mod_{\rm fg}\mc{A}$ is a duality, $\kD(M^{\ast})$ is an injective object of $\Mod_{\rm fg}\mc{A}$.
In particular, $\kD(M^{\ast})$ is a relative-injective object of $\ms{GP}\mc{A}$.
Since $\ms{GP}\mc{A}$ is Frobenius, $\kD(M^{\ast})$ is an object of $\proj\mc{A}$.
Thus we have a functor $\kD\circ(-)^{\ast} : \proj\mc{A} \to \proj\mc{A}$.
This is an equivalence, since its quasi-inverse is given by $(-)^{\ast}\circ\kD$.
Since $\mc{A}$ is idempotent complete, the Yoneda embedding $\mc{A}\to\proj\mc{A}$, $X\mapsto\mc{A}(-,X)$ is equivalence.
Thus there exists an equivalence $\mb{S} : \mc{A} \to \mc{A}$ such that the following diagram is commutative:
$$\xymatrix{
\proj\mc{A} \ar[rr]^{\kD\circ(-)^{\ast}} && \proj\mc{A} \\
\mc{A} \ar[rr]^{\mb{S}} \ar[u]^{\simeq} && \mc{A}\ar[u]^{\simeq}.
}$$
For $X,Y\in\mc{A}$, we have the following isomorphisms which are functorial at $X,Y$:
\begin{align*}
\mc{A}(Y,\mb{S}X) & \simeq \kD(\mc{A}(-,X)^{\ast})(Y)\\
	& \simeq \kD(\Mod\mc{A}(\mc{A}(-,X),\mc{A}(-,Y)))\\
	& \simeq \kD\mc{A}(X,Y).
\end{align*}
This means that $\mb{S}$ is a Serre functor on $\mc{A}$.
\end{proof}
\subsection{Some observations on triangulated categories}
In this subsection, we state some propositions which we use later.
We state one theorem for Frobenius categories.
Let $\mc{A}$ be an additive category and $\mc{B}$ be a subcategory of $\mc{A}$.
For two objects $X,Y\in\mc{A}$, we denote by $\mc{A}_{\mc{B}}(X,Y)$ the subspace of $\mc{A}(X,Y)$ consisting of all morphisms which factor through an object of $\mc{B}$.
We denote by $\mc{A}/[\mc{B}]$ the category defined as follows:
the objects of $\mc{A}/[\mc{B}]$ are the same as $\mc{A}$ and the morphism space is defined by
\begin{align*}
(\mc{A}/[\mc{B}])(X,Y):=\mc{A}(X,Y)/\mc{A}_{\mc{B}}(X,Y),
\end{align*}
for $X,Y\in\mc{A}$.

Let $\mc{F}$ be a Frobenius category, $\mc{P}$ the full subcategory of $\mc{F}$ consisting of the projective objects in $\mc{F}$ and $\U{\mc{F}}:=\mc{F}/[\mc{P}]$.
By Happel \cite{H}, it is known that $\U{\mc{F}}$ is a triangulated category.
Assume that $\mc{P}$ is idempotent complete.
Let $\ms{K}^{\rm b}(\mc{P})$ be the homotopy category of complexes of $\mc{P}$.
We denote by $\ms{K}^{\rm-,b}(\mc{P})$ the full subcategory of $\ms{K}(\mc{P})$ consisting of complexes $X=(X^{i}, d^{i} : X^{i}\to X^{i+1})$ satisfying the following conditions.
\begin{itemize}
\item[{\rm (1)}]
There exists $n_{X}\in\mb{Z}$ such that $X^{i}=0$ for any $i>n_{X}$.
\item[{\rm (2)}]
There exist $m_{X}\in\mb{Z}$ and exact sequences $0\to Y^{i-1} \xto{a^{i-1}} X^{i} \xto{b^{i}} Y^{i} \to 0$  in $\mc{F}$ for any $i\leq m_{X}$ such that $d^{i}=a^{i}b^{i}$ for any $i<m_{X}$.
\end{itemize}
We identify the category $\mc{F}$ with the full subcategory of $\ms{K}^{\rm -,b}(\mc{P})$ consisting of $X$ satisfying $n_{X}\leq 0\leq m_{X}$.
Then we have the following analogy of the well known equivalence due to \cite{Buchweitz, Keller-Vossieck, R}.
\begin{theorem}\cite{IY}\label{rickard-theorem}
Let $\mc{F}$ be a Frobenius category and $\mc{P}$ the full subcategory of $\mc{F}$ consisting of the projective objects.
Assume that $\mc{P}$ is idempotent complete.
Then the composite $\mc{F}\to\ms{K}^{\rm-, b}(\mc{P})\to\ms{K}^{\rm-, b}(\mc{P})/\ms{K}^{\rm b}(\mc{P})$ induces a triangle equivalence $\U{\mc{F}}\xto{\sim} \ms{K}^{\rm-, b}(\mc{P})/\ms{K}^{\rm b}(\mc{P})$.
\end{theorem}
Let $\mc{U}$ be a triangulated category and $\mc{X}$ be a full subcategory of $\mc{U}$.
We call $\mc{X}$ a \emph{thick} subcategory of $\mc{U}$ if $\mc{X}$ is a triangulated subcategory of $\mc{U}$ and closed under direct summands.
We denote by $\thick_{\mc{U}}\mc{X}$ the smallest thick subcategory of $\mc{U}$ which contains $\mc{X}$.
Whenever if there is no danger of confusion, let $\thick_{\mc{U}}\mc{X}=\thick\mc{X}$.
\begin{lemma}\label{basic-triangulated-functor}
Let $\mc{T},\mc{U}$ be triangulated categories and $F : \mc{U}\to\mc{T}$ a triangle functor.
Let $\mc{X}$ be a full subcategory of $\mc{U}$.
Then the following holds.
\begin{itemize}
\item
Assume that a map 
\begin{align*}
F_{M,N[n]} : \mc{U}(M,N) \to \mc{T}(FM,FN[n])
\end{align*}
is an isomorphism for any $M,N\in\mc{X}$ and any $n\in\mb{Z}$.
Then $F : \thick\mc{X}\to\mc{T}$ is fully faithful.
\item
If moreover $\mc{U}$ is idempotent complete, $\thick\mc{X}=\mc{U}$ and $\thick(\Im(F))=\mc{T}$, then $F$ is an equivalence.
\end{itemize}
\end{lemma}
\section{Repetitive categories}\label{section-repetitive}
\subsection{Repetitive categories}\label{subsection-repetitive-dualizing}
We recall the definition of repetitive categories of additive categories.
The aim of this subsection is to show Theorem \ref{repetitive-dualizing}.
\begin{definition}\label{defition-repetitive}
Let $\mc{A}$ be a $k$-linear additive category.
The \emph{repetitive category} $\ms{R}\mc{A}$ is the $k$-linear additive category generated by the following category:
the class of objects is $\{ (X,i) \mid X \in\mc{A}, i\in\mb{Z} \}$ and the morphism space is given by 
	\begin{align*}
	\ms{R}\mc{A}\bigl((X,i), (Y,j)\bigr)=\begin{cases}
					\mc{A}(X,Y) & i=j, \\
					\kD\mc{A}(Y,X) & j=i+1, \\
					0 & \textnormal{else}.
				\end{cases}
	\end{align*}
For $f\in\ms{R}\mc{A}\bigl((X,i), (Y,j)\bigr)$ and $g\in\ms{R}\mc{A}\bigl((Y,j), (Z,k)\bigr)$, the composition is given by 
	\begin{align*}
	g\circ f=\begin{cases}
					g\circ f & i=j=k, \\
					\bigl(\kD\mc{A}(Z,f)\bigr)(g) & i=j=k-1, \\
					\bigl(\kD\mc{A}(g,X)\bigr)(f) & i+1=j=k, \\
					0 & \textnormal{else}.
				\end{cases}
	\end{align*}
\end{definition}
We describe fundamental properties of repetitive categories of $\Hom$-finite categories.
\begin{lemma}\label{properties-of-repetitive}
Let $\mc{A}$ be a $k$-linear, Hom-finite additive category.
The following statements hold.
\begin{itemize}
\item[{\rm (a)}]
$\ms{R}\mc{A}$ is $\Hom$-finite.
\item[{\rm (b)}]
$\ms{R}\mc{A}$ has a Serre functor $\mb{S}$ which is defined by $\mb{S}(X,i):=(X,i+1)$.
\item[{\rm (c)}]
If $\mc{A}$ is idempotent complete, then so is $\ms{R}\mc{A}$.
\end{itemize}
\end{lemma}
\begin{proof}
(a) (b) These are clear by the definition.

(c) By the definition, an object of $\ms{R}\mc{A}$ is indecomposable if and only if it is isomorphic to an object $(X,i)$,
where $X$ is an indecomposable object of $\mc{A}$ and $i$ is some integer.
Let $X$ be an indecomposable object of $\mc{A}$ and $i$ be an integer.
Since $\mc{A}$ is idempotent complete and Proposition \ref{prop-Krull-Schmidt}, $\End_{\ms{R}\mc{A}}(X,i)=\End_{\mc{A}}(X)$ is local.
Therefore again by Proposition \ref{prop-Krull-Schmidt}, $\ms{R}\mc{A}$ is idempotent complete.
\end{proof}
We see a relation between the categories $\mod\mc{A}$ and $\mod\ms{R}\mc{A}$ and consequently, we show Theorem \ref{repetitive-dualizing}.
Let $\mc{A}$ be a $k$-linear additive category and $i\in\mb{Z}$.
Put the following full subcategory of $\ms{R}\mc{A}$:
\begin{align*}
\mc{A}_{i}:=\add\{\, (X,i) \in\ms{R}\mc{A} \mid X\in\mc{A} \,\}.
\end{align*}
An inclusion functor $\mc{A}_{i} \to \ms{R}\mc{A}$ induces an exact functor
\begin{align*}
\rho_{i} : \Mod\ms{R}\mc{A}\to\Mod\mc{A}_{i}.
\end{align*}
Since a functor $\mc{A}\to\mc{A}_{i}$ defined by $X \mapsto (X,i)$ is an equivalence, we denote an object $(X,i)$ of $\mc{A}_{i}$ by $X$ for simplicity.

Since we have a full dense functor $\ms{R}\mc{A} \to \mc{A}_{i}$ given by $(X,j)\mapsto X$ if $j=i$ and $(X,j)\mapsto0$ if else,
we have a fully faithful functor from $\Mod\mc{A}_{i}$ to $\Mod\ms{R}\mc{A}$.
Therefore we identify $\Mod\mc{A}_{i}$ with the full subcategory of $\Mod\ms{R}\mc{A}$ consisting of $\ms{R}\mc{A}$-modules $M$ such that $M(X,j)=0$ for any $j\neq i$ and any $X\in\mc{A}$.
\begin{lemma}\label{embedding-composition}
Let $\mc{A}$ be an additive category and $i,j\in\mb{Z}$.
\begin{itemize}
\item[{\rm (a)}]
We have $\rho_{j}|_{\Mod\mc{A}_{i}}={\rm id}_{\Mod\mc{A}_{i}}$ if $j=i$ and $\rho_{j}|_{\Mod\mc{A}_{i}}=0$ if else.

\item[{\rm (b)}]
For any $X\in\mc{A}$, we have an exact sequence
\begin{align}\label{DA-RA-A}
0 \to \kD\mc{A}_{i-1}(X,-)\xto{\beta}\ms{R}\mc{A}(-,(X,i))\xto{\alpha}\mc{A}_{i}(-,X)\to 0
\end{align}
in $\Mod\ms{R}\mc{A}$.
In particular, we have $\rho_{j}(P)\in\add\{ \mc{A}_{j}(-,X), \kD\mc{A}_{j}(X,-) \mid X\in\mc{A} \}$ for any $P\in\proj\ms{R}\mc{A}$ and $j\in\mb{Z}$.

\item[{\rm (c)}]
Each finitely generated $\mc{A}_{i}$-module is a finitely generated $\ms{R}\mc{A}$-module.
\end{itemize}
\end{lemma}
\begin{proof}
(a)
The assertions follow from the definition of $\rho_{j}$.

(b)
We construct morphisms $\alpha, \beta$ in $\Mod\ms{R}\mc{A}$.
For an object $(Y,j)$ of $\ms{R}\mc{A}$,
define
\begin{align*}
	\alpha_{(Y,j)}:=\begin{cases}
					{\rm id}_{\mc{A}(Y,X)} & j=i, \\
					0 & \textnormal{else},
				\end{cases}
	\quad
	\beta_{(Y,j)}:=\begin{cases}
					{\rm id}_{\kD\mc{A}(X,Y)} & j+1=i, \\
					0 & \textnormal{else},
				\end{cases}
	\end{align*}
and extend $\alpha$ and $\beta$ on $\ms{R}\mc{A}$ additively.
We can show that $\alpha$ and $\beta$ are actually morphisms in $\Mod\ms{R}\mc{A}$.
By definitions of $\alpha$ and $\beta$, for an object $(Y,j)$ of $\ms{R}\mc{A}$, we have the following exact sequence
$$0\to\kD\mc{A}_{i-1}(X,(Y,j)) \xto{\beta_{(Y,j)}} \ms{R}\mc{A}((Y,j),(X,i))\xto{\alpha_{(Y,j)}}\mc{A}_{i}((Y,j),X)\to 0$$
in $\Mod k$.
Thus we have an exact sequence (\ref{DA-RA-A}).
Since $\rho_{j}$ is exact, by applying $\rho_{j}$ to the exact sequence (\ref{DA-RA-A}) and by using (a), we have the assertion.

(c)
This follows from (b).
\end{proof}
By the following lemma, we construct a filtration of a module over repetitive categories.
For $M\in\Mod\ms{R}\mc{A}$, put $\supp M:=\{\, i\in\mb{Z} \mid \rho_{i}(M)\neq 0 \,\}$.
\begin{lemma}\label{filtration-module}
Let $M\in\Mod\ms{R}\mc{A}$ and $i\in\mb{Z}$.
\begin{itemize}
\item[{\rm (a)}]
If $\rho_{i-1}(M)=0$, then there exists a short exact sequence
\begin{align*}
0\to \rho_{i}(M) \xto{\alpha} M \to N \to 0
\end{align*}
in $\Mod\ms{R}\mc{A}$ such that $\rho_{i}(N)=0$ and $\rho_{j}(N)=\rho_{j}(M)$ for any $j>i$.

\item[{\rm (b)}]
Assume that $\supp M$ is a finite set and put $m:=\max\supp M$ and $n:=\min\supp M$.
Then there exists a sequence of subobjects of $M$:
\begin{align*}
0=M_{n-1}\subset M_{n} \subset \cdots \subset M_{m-1}\subset M_{m}=M
\end{align*}
such that $M_{i}/M_{i-1}\simeq \rho_{i}(M)$ for any $i=n, n+1,\ldots,m$.
\end{itemize}
\end{lemma}
\begin{proof}
(a)
We construct a monomorphism $\alpha : \rho_{i}(M) \to M$ in $\Mod\ms{R}\mc{A}$.
For an object $(X,j)$ of $\ms{R}\mc{A}$, define
\begin{align*}
	\alpha_{(X,j)}:=\begin{cases}
					{\rm id}_{M(X,j)} & j=i, \\
					0 & \textnormal{else},
				\end{cases}
\end{align*}
and extend this on $\ms{R}\mc{A}$ additively.
Since $\rho_{i-1}(M)=0$, $\alpha$ is a morphism of $\Mod\ms{R}\mc{A}$.
By the definition, $\alpha$ is mono.
Then we have an exact sequence $0 \to \rho_{i}(M) \to M \to N \to 0$ in $\Mod\ms{R}\mc{A}$, where $N:=\Cok(\alpha)$.
By Lemma \ref{embedding-composition}, we have $\rho_{j}(\rho_{i}(M))=\rho_{i}(M)$ if $j=i$ and $\rho_{j}(\rho_{i}(M))=0$ if else.
Therefore by applying the functor $\rho_{j}$ to this exact sequence, we have the assertion.

(b)
This follows from (a).
\end{proof}
By the following two lemmas, we see that the functors $\Mod\mc{A}\to\Mod\ms{R}\mc{A}$ and $\rho_{i} : \Mod\ms{R}\mc{A}\to\Mod\mc{A}$ restrict to functors between $\mod\mc{A}$ and $\mod\ms{R}\mc{A}$ under certain assumptions.
For simplicity, we use the notation $\mod_{-1}\mc{A}:=\Mod\mc{A}$, $\mod_{\infty}\mc{A}:=\mod\mc{A}$ and $\infty-1:=\infty$.
\begin{lemma}\label{modA-to-modRA}
Let $\mc{A}$ be a $k$-linear, Hom-finite additive category and $n\in\mb{Z}_{\geq 0}\cup\{\infty\}$.
Assume that $\kD\mc{A}(X,-)\in\mod_{n-1}\mc{A}$ holds for any $X\in\mc{A}$.
Then an inclusion functor $\Mod\mc{A}_{i}\to\Mod\ms{R}\mc{A}$ restricts to a functor $\mod_{n}\mc{A}_{i}\to\mod_{n}\ms{R}\mc{A}$ for any $i\in\mb{Z}$.
\end{lemma}
\begin{proof}
Let $n\in\mb{Z}_{\geq 0}$.
It is sufficient to show that $\mc{A}_{i}(-,X)\in\mod_{n}\ms{R}\mc{A}$ for any $i\in\mb{Z}$.
In fact, any $M\in\mod_{n}\mc{A}_{i}$ has an exact sequence $P_{n} \to \cdots \to P_{0} \to M \to 0$ with $P_{i}\in\proj\mc{A}_{i}$ and hence $M$ belongs to $\mod_{n}\ms{R}\mc{A}$ by Lemma \ref{M-in-mod-n-A} (a).

We show $\proj\mc{A}_{i}\subset\mod_{n}\ms{R}\mc{A}$ for any $i\in\mb{Z}$ by an induction on $n$.
If $n=0$, then by Lemma \ref{embedding-composition} (c), we have the assertion.
Let $n>0$, $X\in\mc{A}$ and $i\in\mb{Z}$.
By Lemma \ref{embedding-composition} (b), 
there exists an exact sequence
\begin{align*}
0 \to \kD\mc{A}_{i-1}(X,-) \to \ms{R}\mc{A}(-,(X,i)) \to \mc{A}_{i}(-,X) \to 0.
\end{align*}
By the inductive hypothesis, $\kD\mc{A}_{i-1}(X,-)\in\mod_{n-1}\ms{R}\mc{A}$ holds.
Therefore we have $\mc{A}_{i}(-,X)\in\mod_{n}\ms{R}\mc{A}$ by Lemma \ref{M-in-mod-n-A} (c).

By an argument similar to the above, the assertion holds when $n=\infty$.
\end{proof}
\begin{lemma}\label{modRA-to-modA}
Let $\mc{A}$ be a $k$-linear, Hom-finite additive category, $n\in\mb{Z}_{\geq 0}\cup\{\infty\}$.
Assume that $\kD\mc{A}(X,-)\in\mod_{n}\mc{A}$ holds for any $X\in\mc{A}$.
Then the functor $\rho_{i} : \Mod\ms{R}\mc{A}\to\Mod\mc{A}_{i}$ restricts to a functor $\mod_{n}\ms{R}\mc{A}\to\mod_{n}\mc{A}_{i}$ for any $i\in\mb{Z}$.
\end{lemma}
\begin{proof}
Let $n\in\mb{Z}_{\geq 0}$ and $M\in\mod_{n}\ms{R}\mc{A}$.
We have an exact sequence $P_{n}\to \cdots \to P_{1}\to P_{0}\to M\to0$ in $\Mod\ms{R}\mc{A}$, where $P_{j}\in\proj\ms{R}\mc{A}$ for each $j\geq 0$.
Since $\rho_{i}$ is exact, we have an exact sequence $\rho_{i}(P_{n}) \to \cdots \to\rho_{i}(P_{1})\to \rho_{i}(P_{0})\to \rho_{i}(M)\to0$ in $\Mod\mc{A}_{i}$.
By the assumption and Lemma \ref{embedding-composition} (b), $\rho_{i}(P_{j})\in\mod_{n}\mc{A}_{i}$ holds for any $j\geq 0$.
Therefore $\rho_{i}(M)\in\mod_{n}\mc{A}_{i}$ holds by Lemma \ref{M-in-mod-n-A} (a).

By an argument similar to the above, the assertion holds when $n=\infty$.
\end{proof}
Note that in general $\mod\ms{R}\mc{A}=\mod_{1}\ms{R}\mc{A}$ does not hold for a $k$-linear additive category $\mc{A}$.
This is the case where $\mc{A}$ is a dualizing $k$-variety by Theorem \ref{repetitive-dualizing} below.
Note that there exists an equivalence $(\ms{R}\mc{A})^{\rm op}\simeq \ms{R}(\mc{A}^{\rm op})$ given by $(X,i) \mapsto (X,-i)$.
\begin{theorem}\label{repetitive-dualizing}
Let $\mc{A}$ be a dualizing $k$-variety.
Then the following statements hold.
\begin{itemize}
\item[{\rm (a)}]
$\ms{R}\mc{A}$ and $(\ms{R}\mc{A})^{\rm op}$ have weak kernels.
\item[{\rm (b)}]
$\ms{R}\mc{A}$ is a dualizing $k$-variety.
\end{itemize}
\end{theorem}
\begin{proof}
Note that since $\mc{A}$ is a dualizing $k$-variety, $\kD\mc{A}(-,X)\in\mod_{1}\mc{A}$ holds for any $X\in\mc{A}$ and $\mod_{1}\mc{A}=\mod\mc{A}$ holds.

(a)
Let $\ms{X}, \ms{Y}\in\ms{R}\mc{A}$ and $f : \ms{R}\mc{A}(-,\ms{X})\to\ms{R}\mc{A}(-,\ms{Y})$ be a morphism of $\mod\ms{R}\mc{A}$.
We show that $K:=\Ker(f)$ is a finitely generated $\ms{R}\mc{A}$-module.
For any $i\in\mb{Z}$, we have an exact sequence $0\to \rho_{i}(K) \to \rho_{i}(\ms{R}\mc{A}(-,\ms{X}))\to\rho_{i}(\ms{R}\mc{A}(-,\ms{Y}))$ in $\Mod\mc{A}_{i}$.
By Lemma \ref{modRA-to-modA},  we have $\rho_{i}(\ms{R}\mc{A}(-,\ms{X})), \rho_{i}(\ms{R}\mc{A}(-,\ms{Y}))\in\mod\mc{A}_{i}$.
Therefore $\rho_{i}(K)\in\mod\mc{A}_{i}$ for any $i\in\mb{Z}$, since $\mc{A}_{i}\simeq\mc{A}$ is a dualizing $k$-variety.
By Lemma \ref{modA-to-modRA}, $\rho_{i}(K)\in\mod\ms{R}\mc{A}$ for any $i\in\mb{Z}$.
Since $K$ is a submodule of $\ms{R}\mc{A}(-,\ms{X})$, $\supp K$ is a finite set.
Thus by Lemma \ref{filtration-module} (b), $K$ has a finite filtration by finitely presented $\ms{R}\mc{A}$-modules $\{\rho_{i}(K)\mid i\in\mb{Z}\}$ and we have $K\in\mod\ms{R}\mc{A}$.
In particular, $K$ is finitely generated and $\ms{R}\mc{A}$ has weak kernels.
Since $(\ms{R}\mc{A})^{\rm op}\simeq \ms{R}(\mc{A}^{\rm op})$ holds and $\mc{A}^{\rm op}$ is a dualizing $k$-variety, $(\ms{R}\mc{A})^{\rm op}$ has weak kernels.

(b)
By the definition of dualizing $k$-varieties, $\mc{A}$ is $\Hom$-finite and idempotent complete.
By Lemma \ref{properties-of-repetitive}, $\ms{R}\mc{A}$ is $\Hom$-finite and idempotent complete with a Serre functor.
Therefore by Proposition \ref{dualizing-Serre-Frobenius}, $\ms{R}\mc{A}$ is a dualizing $k$-variety.
\end{proof}
\subsection{Tilting subcategories}\label{subsection-happel-theorem-dualizing}
The aim of this subsection is to show Theorem \ref{thm-tilting-subcategory}.
Before stating the main theorem, we need the following definition.

Let $\mc{A}$ be a $k$-linear, Hom-finite additive category.
We denote by \[\rho : \Mod\ms{R}\mc{A}\to\Mod\mc{A}\] the forgetful functor, that is, $\rho(M):=\bigoplus_{i\in\mb{Z}}\rho_{i}(M)$ for any $M\in\Mod\ms{R}\mc{A}$, where we regard an $\mc{A}_{i}$-module $\rho_{i}(M)$ as an $\mc{A}$-module by the equivalence $\Mod\mc{A}_{i}\simeq\Mod\mc{A}$.
Note that $\rho$ is an exact functor.
We denote by $\ms{GP}(\ms{R}\mc{A},\mc{A})$ the full subcategory of $\ms{GP}(\ms{R}\mc{A})$ consisting of all objects $M$ such that the projective dimension of $\rho(M)$ over $\mc{A}$ is finite, that is,
\[\ms{GP}(\ms{R}\mc{A},\mc{A}):=\{\,M\in\ms{GP}(\ms{R}\mc{A}) \mid \pd_{\mc{A}}\rho (M)<\infty\,\}.\]
We consider the following condition on $\mc{A}$:
\begin{center}
(G) : the projective dimension of $\kD\mc{A}(X,-)$ over $\mc{A}$ is finite for any $X\in\mc{A}$.
\end{center}
\begin{proposition}\label{prop-A-G}
Let $\mc{A}$ be a $k$-linear, Hom-finite additive category.
Then $\mc{A}$ satisfies (G) if and only if $\proj\ms{R}\mc{A}\subset\ms{GP}(\ms{R}\mc{A},\mc{A})$ holds.
In this case, the following statements fold.
\begin{itemize}
\item[{\rm (a)}]
$\ms{GP}(\ms{R}\mc{A},\mc{A})$ is a Frobenius category such that the projective objects is the objects of $\proj\ms{R}\mc{A}$.

\item[{\rm (b)}]
The inclusion functor $\ms{GP}(\ms{R}\mc{A},\mc{A})\to\ms{GP}(\ms{R}\mc{A})$ induces a fully faithful triangle functor $\U{\ms{GP}}(\ms{R}\mc{A},\mc{A})\to\U{\ms{GP}}(\ms{R}\mc{A})$.
\end{itemize}
\end{proposition}
\begin{proof}
The first assertion follows from Lemma \ref{embedding-composition} (b).
Assume that $\mc{A}$ satisfies (G).

(a) By the definition and since $\rho$ is exact, $\ms{GP}(\ms{R}\mc{A},\mc{A})$ is extension closed subcategory of $\Mod\ms{R}\ms{A}$ and has enough projectives and enough injectives.
Clearly, an object of $\proj\ms{R}\mc{A}$ is relative projective of $\ms{GP}(\ms{R}\mc{A},\mc{A})$.
Let $Q$ be a relative projective object of $\ms{GP}(\ms{R}\mc{A},\mc{A})$.
There exists an exact sequence $0\to M \to P \to Q \to 0$ in $\ms{GP}(\ms{R}\mc{A})$ with $P\in\proj\ms{R}\mc{A}$.
We have $M\in\ms{GP}(\ms{R}\mc{A},\mc{A})$ and therefore this sequence splits.
Consequently, the relative projective objects of $\ms{GP}(\ms{R}\mc{A},\mc{A})$ is the objects of $\proj\ms{R}\mc{A}$.

(b) This follows from (a).
\end{proof}
We regard $\U{\ms{GP}}(\ms{R}\mc{A},\mc{A})$ as a thick subcategory of $\U{\ms{GP}}(\ms{R}\mc{A})$ by Proposition \ref{prop-A-G} (b) if $\mc{A}$ satisfies (G).
Let $\mc{A}$ be a $k$-linear, Hom-finite additive category.
We consider the following condition on $\mc{A}$:
\begin{center}
(IFP) : $\kD\mc{A}(X,-)\in\mod\mc{A}$ holds for any $X\in\mc{A}$.
\end{center}
Note that if $\mc{A}$ is a dualizing $k$-variety, then $\mc{A}$ satisfies (IFP).
We denote by $\mc{M}$ the full subcategory of $\Mod\ms{R}\mc{A}$ given by
\begin{align*}
\mc{M}:=\add\{\, \mc{A}_{0}(-,X) \mid X\in\mc{A}\,\}.
\end{align*}

We recall the definition of tilting subcategories of a triangulated category.
\begin{definition}\label{definition-tilting-subcategory}
Let $\mc{T}$ be a triangulated category.
A full subcategory $\mc{M}$ of $\mc{T}$ is called a \emph{tilting subcategory} of $\mc{T}$ if $\mc{T}(\mc{M},\mc{M}[i])=0$ for any $i\neq 0$ and $\thick\mc{M}=\mc{T}$.
\end{definition}
We establish the following result.
\begin{theorem}\label{thm-tilting-subcategory}
Let $\mc{A}$ be a $k$-linear, Hom-finite additive category and assume that  $\mc{A}$ and $\mc{A}^{\rm op}$ satisfy (IFP).
Then the following holds.
\begin{itemize}
\item[{\rm (a)}]
If $\mc{A}$ and $\mc{A}^{\rm op}$ satisfy (G),
then $\mc{M}\subset\ms{GP}(\ms{R}\mc{A},\mc{A})$ holds and $\mc{M}$ gives a tilting subcategory of $\U{\ms{GP}}(\ms{R}\mc{A},\mc{A})$.

\item[{\rm (b)}]
If each object of $\mod\mc{A}$ and $\mod\mc{A}^{\rm op}$ has finite projective dimension, then
$\mc{M}\subset\ms{GP}(\ms{R}\mc{A})$ holds and $\mc{M}$ gives a tilting subcategory of $\U{\ms{GP}}(\ms{R}\mc{A})$.
\end{itemize}
\end{theorem}
In the case where $\mc{A}$ is a dualizing $k$-variety, we have the following corollary.
\begin{corollary}\label{cor-tilting-subcategory-dualizing}
Let $\mc{A}$ be a dualizing $k$-variety.
If each object of $\mod\mc{A}$ and $\mod\mc{A}^{\rm op}$ has finite projective dimension, then $\mc{M}$ is a tilting subcategory of $\smod\ms{R}\mc{A}$.
\end{corollary}
Before starting the proof of Theorem \ref{thm-tilting-subcategory}, we prepare two lemmas.
Let $\mc{A}$ be a $k$-linear additive category and $i\in\mb{Z}$.
Put the following full subcategories of $\ms{R}\mc{A}$:
\begin{align*}
\mc{A}_{<i}:=\bigvee_{j<i}\mc{A}_{j}, \quad
\mc{A}_{\geq i}:=\bigvee_{j\geq i}\mc{A}_{j}.
\end{align*}
For $M\in\Mod\ms{R}\mc{A}$ and $i\in\mb{Z}$, let $\rho_{<i}(M):=\bigoplus_{j<i}\rho_{j}(M)$ and $\rho_{\geq i}(M):=\bigoplus_{j\geq i}\rho_{j}(M)$.
\begin{lemma}\label{projective-cover}
Let $\mc{A}$ be a $k$-linear, Hom-finite additive category.
Let $M$ and $N$ be finitely generated $\ms{R}\mc{A}$-modules and $i\in\mb{Z}$.
Assume that $\rho_{\geq i}(M)=0$ and $\rho_{<i}(N)=0$.
\begin{itemize}
\item[{\rm (a)}]
There exist epimorphisms 
\begin{align*}
\ms{R}\mc{A}(-,\ms{X})\to M, \quad \ms{R}\mc{A}(-,\ms{Y})\to N, 
\end{align*}
for some $\ms{X}\in\mc{A}_{<i}$ and $\ms{Y}\in\mc{A}_{\geq i}$.

\item[{\rm (b)}]
We have $(\Mod\ms{R}\mc{A})(M,N)=0$ and $(\Mod\ms{R}\mc{A})(N,M)=0$.

\item[{\rm (c)}]
Assume $M\in\mod\ms{R}\mc{A}$.
Let
\begin{align}\label{projective-resolution}
\cdots\to P_{2}\xto{f_{2}} P_{1}\xto{f_{1}} P_{0}\xto{f_{0}} M\to0
\end{align}
 be a minimal projective resolution of $M$ in $\mod\ms{R}\mc{A}$.
Then we have $\rho_{\geq i}(\Ker f_{l})=0$ for $l\geq 0$.
Moreover by applying a functor $\rho_{i-1}$, we have a minimal projective resolution of $\rho_{i-1}(M)$ in $\mod\mc{A}_{i-1}$.
\end{itemize}
\end{lemma}
\begin{proof}
(a)
Since $M$ and $N$ are finitely generated,  there exist epimorphisms $\ms{R}\mc{A}(-,\ms{X})\to M$ and $\ms{R}\mc{A}(-,\ms{Y})\to N$, where $\ms{X}$ and $\ms{Y}$ are in $\ms{R}\mc{A}$.
Let $\ms{W}$ be an object of $\mc{A}_{\geq i}$.
By Yoneda's lemma and the assumption, we have $(\Mod\ms{R}\mc{A})(\ms{R}\mc{A}(-,\ms{W}),M)\simeq M(\ms{W})=0$.
Therefore we can replace $\ms{X}$ with an object of $\mc{A}_{<i}$.
Similarly, we can replace $\ms{Y}$ with an object of $\mc{A}_{\geq i}$.

(b)
By (a), there exists an epimorphism $\ms{R}\mc{A}(-,\ms{X})\to M$, where $\ms{X}\in\mc{A}_{<i}$.
We have a monomorphism $(\Mod\ms{R}\mc{A})(M,N) \to (\Mod\ms{R}\mc{A})(\ms{R}\mc{A}(-,\ms{X}),N)$.
Since $(\Mod\ms{R}\mc{A})(\ms{R}\mc{A}(-,\ms{X}),$ $N)\simeq N(\ms{X})=0$, $(\Mod\ms{R}\mc{A})(M,N)=0$ holds.
Similarly, by applying $(\Mod\ms{R}\mc{A})(-,M)$ to an epimorphism $\ms{R}\mc{A}(-,\ms{Y})\to N$, we have $(\Mod\ms{R}\mc{A})(N,M)=0$.

(c)
By (a), there exists $\ms{X}_{0}\in\mc{A}_{<i}$ such that $P_{0}$ is a direct summands of $\ms{R}\mc{A}(-,\ms{X}_{0})$.
We have $\rho_{\geq i}(\ms{R}\mc{A}(-,\ms{X}_{0}))=0$.
Therefore the submodule $\Ker f_{0}$ of $\ms{R}\mc{A}(-,\ms{X}_{0})$ satisfies $\rho_{\geq i}(\Ker f_{0})=0$.
By using this argument inductively, we have that there exist $\ms{X}_{l}\in\mc{A}_{<i}$ such that $P_{l}$ is a direct summands of $\ms{R}\mc{A}(-,\ms{X}_{l})$ for any $l\geq 0$.
Therefore we have $\rho_{\geq i}(\Ker f_{l})=0$ for $l\geq 0$.

For any $l\geq 0$, by Lemma \ref{embedding-composition}, $\rho_{i-1}(P_{l})$ is a direct sum of $\mc{A}_{i-1}(-,X)$ for some $X\in\mc{A}$ and zero objects.
Therefore each $\rho_{i-1}(P_{l})$ is a projective $\mc{A}_{i-1}$-module.
Minimality comes from the minimality of the resolution (\ref{projective-resolution}).
\end{proof}
We see when $\ms{GP}(\ms{R}\mc{A})$ contains the representable functors on $\mc{A}$.
Note that there exists an equivalence $(\ms{R}\mc{A})^{\rm op}\simeq \ms{R}(\mc{A}^{\rm op})$ given by $(X,i) \mapsto (X,-i)$.
Thus we have a duality \[\Mod_{\rm fg}\ms{R}\mc{A} \xto{\kD} \Mod_{\rm fg}(\ms{R}\mc{A})^{\rm op}\xto{\sim} \Mod_{\rm fg}\ms{R}(\mc{A}^{\rm op}).\]
By this duality, a full subcategory $\mod\mc{A}_{i}$ of $\mod\ms{R}\mc{A}$ goes to a full subcategory $\mod(\mc{A}^{\rm op})_{-i}$ of  $\mod\ms{R}(\mc{A}^{\rm op})$.
\begin{lemma}\label{lem-rep-fun-GP}
Let $\mc{A}$ be a $k$-linear, Hom-finite additive category.
\begin{itemize}
\item[{\rm (a)}]
The following statements are equivalent.
\begin{itemize}
\item[{\rm (i)}]
$\mc{A}$ and $\mc{A}^{\rm op}$ satisfy (IFP).
\item[{\rm (ii)}]
$\mc{A}_{i}(-,X)\in\ms{GP}(\ms{R}\mc{A})$ and $\mc{A}_{i}(X,-)\in\ms{GP}(\ms{R}\mc{A})^{\rm op}$ hold  for any $X\in\mc{A}$ and $i\in\mb{Z}$.
\item[{\rm (iii)}]
$\kD\mc{A}_{i}(X,-)\in\ms{GP}(\ms{R}\mc{A})$ and $\kD\mc{A}_{i}(-,X)\in\ms{GP}(\ms{R}\mc{A})^{\rm op}$ hold  for any $X\in\mc{A}$ and $i\in\mb{Z}$.
\end{itemize}
\item[{\rm (b)}]
If $\mc{A}$ and $\mc{A}^{\rm op}$ satisfy (IFP), then
$\rho_{i}(M)\in\ms{GP}(\ms{R}\mc{A})$ holds for any $M\in\ms{GP}(\ms{R}\mc{A})$ and $i\in\mb{Z}$.
\end{itemize}
\end{lemma}
\begin{proof}
Note that by Lemma \ref{properties-of-repetitive}, $\ms{R}\mc{A}$ has a Serre functor $\mb{S}$.
Thus by Lemma \ref{lem-Serre-Adual}, we have an isomorphism of functors $(-)^{\ast}\simeq \kD(-\circ\mb{S}^{-1}) : \Mod_{\rm fg}\ms{R}\mc{A} \to \Mod_{\rm fg}\ms{R}(\mc{A}^{\rm op})$.
We have 
\begin{align}\label{RA-dual-of-A}
(\mc{A}_{i}(-,X))^{\ast}\simeq\kD(\mc{A}^{\rm op})_{-i-1}(X,-)=\kD\mc{A}_{-i-1}(-,X)
\end{align}
for any $X\in\mc{A}$ and $i\in\mb{Z}$.
Therefore (ii) and (iii) of (a) are equivalent.

(a)
We show that (i) implies (ii).
Let $X\in\mc{A}$.
By Lemma \ref{modA-to-modRA}, $\mc{A}_{i}(-,X)\in\mod\ms{R}\mc{A}$ holds.
We have $(\mc{A}_{i}(-,X))^{\ast}\in\mod(\ms{R}\mc{A})^{\rm op}$, by the equality (\ref{RA-dual-of-A}) and Lemma \ref{modA-to-modRA}.
Therefore by Lemma \ref{lem-Serre-Adual} (b), we have $\mc{A}_{i}(-,X)\in\ms{GP}(\ms{R}\mc{A})$.
Dually, we have $\mc{A}_{i}(X,-)\in\ms{GP}(\ms{R}\mc{A})^{\rm op}$.

We show that (ii) implies (i).
Let $X\in\mc{A}$.
Take a minimal projective resolution of $\mc{A}_{i}(-,X)$ in $\mod\ms{R}\mc{A}$:
\begin{align*}
\cdots \to Q_{2} \to Q_{1} \xto{d_{1}} \ms{R}\mc{A}(-,(X,i)) \to \mc{A}_{i}(-,X) \to 0.
\end{align*}
By Lemma \ref{embedding-composition} (b), we have $\Im d_{1}=\kD\mc{A}_{i-1}(X,-)$.
By Lemma \ref{projective-cover} (c), applying $\rho_{i-1}$, we have $\kD\mc{A}_{i-1}(X,-)\in\mod\mc{A}_{i-1}$.
This means $\kD\mc{A}(X,-)\in\mod\mc{A}$.
Dually, we have $\kD\mc{A}(-,X)\in\mod\mc{A}^{\rm op}$.

(b)
By Lemma \ref{embedding-composition} (b), we have $\rho_{i}(P)\in\add\{ \mc{A}_{i}(-,X), \kD\mc{A}_{i}(X,-) \mid X\in\mc{A} \}$ for any $P\in\proj\ms{R}\mc{A}$.
Therefore $(\rho_{i}(P))^{\ast}\in\mod(\mc{A}^{\rm op})_{-i-1}$ holds by the equality (\ref{RA-dual-of-A}) and the assumption.
Let $M\in\ms{GP}(\ms{R}\mc{A})$ and $P_{\bullet}=(P_{j}, d_{j}: P_{j}\to P_{j+1})$ be a totally acyclic complex such that $\Im d_{0}=M$, where $P_{j}\in\proj\ms{R}\mc{A}$.
By applying $\rho_{i}$, we have an exact sequence $\rho_{i}(P_{\bullet})=(\rho_{i}(P_{j}), \rho_{i}(d_{j}):  \rho_{i}(P_{j})\to  \rho_{i}(P_{j+1}))$ such that $\Im \rho_{i}(d_{0})=\rho_{i}(M)$.
We have an exact sequence $\cdots \to \rho_{i}(P_{-1})\to \rho_{i}(P_{0}) \to \rho_{i}(M) \to 0$.
By Lemmas \ref{M-in-mod-n-A} (b) and \ref{modA-to-modRA}, $\rho_{i}(M)\in\mod\ms{R}\mc{A}$ holds.
By applying a functor $(-)^{\ast}$ to $0\to \rho_{i}(M) \to \rho_{i}(P_{1}) \to \rho_{i}(P_{2}) \to \cdots$, and using Lemma \ref{M-in-mod-n-A} (b) to the resulting exact sequence, we have $(\rho_{i}(M))^{\ast}\in\mod(\ms{R}\mc{A})^{\rm op}$.
Therefore we have $\rho_{i}(M)\in\ms{GP}(\ms{R}\mc{A})$ by Lemma \ref{lem-Serre-Adual} (b).
\end{proof}
By Lemma \ref{lem-rep-fun-GP}, if $\mc{A}$ and $\mc{A}^{\rm op}$ satisfy (IFP), then $\mc{M}\subset\ms{GP}(\ms{R}\mc{A})$ holds.
We also denote by $\mc{M}$ the subcategory of $\U{\ms{GP}}(\ms{R}\mc{A})$ consisting of objects $\mc{A}_{0}(-,X)$ for any $X\in\mc{A}$.
Then we show Theorem \ref{thm-tilting-subcategory}.
We divide the proof into two propositions.
Put $\mc{T}:=\U{\ms{GP}}(\ms{R}\mc{A})$.
\begin{proposition}\label{condition-orthogonal}
Let $\mc{A}$ be a $k$-linear, Hom-finite additive category and assume that $\mc{A}$ and $\mc{A}^{\rm op}$ satisfy (IFP).
Then we have $\mc{T}(\mc{M},\mc{M}[i])=0$ for any $i\neq 0$.
\end{proposition}
\begin{proof}

Let $X\in\mc{A}$ and 
\begin{align*}
\cdots\to P_{2}\xto{f_{2}} P_{1}\xto{f_{1}} P_{0}\xto{f_{0}} \mc{A}_{0}(-,X)\to0
\end{align*}
be a minimal projective resolution in $\mod\ms{R}\mc{A}$.
Put $K^{i}:=\Ker(f^{i-1})$ for $i\geq 1$.
By Lemmas \ref{embedding-composition} (b) and \ref{projective-cover} (c), we have $\rho_{\geq 0}(K^{i})=0$ for $i\geq 1$.
Let $Y\in\mc{A}$.
Since $\rho_{<0}(\mc{A}_{0}(-,Y))=0$ and Lemma \ref{projective-cover} (b), we have
\begin{align*}
(\Mod\ms{R}\mc{A})(K^{i},\mc{A}_{0}(-,Y))=0, \quad (\Mod\ms{R}\mc{A})(\mc{A}_{0}(-,Y),K^{i})=0,
\end{align*}
for any $i\geq 1$.
Therefore we have
\begin{align*}
\mc{T}(\mc{A}_{0}(-,Y),\mc{A}_{0}(-,X)[-i])=\mc{T}(\mc{A}_{0}(-,Y),K^{i})=0,\\
\mc{T}(\mc{A}_{0}(-,X),\mc{A}_{0}(-,Y)[i])=\mc{T}(K^{i},\mc{A}_{0}(-,Y))=0,
\end{align*}
for any $i\geq 1$.
\end{proof}
\begin{proposition}\label{condition-thick}
Let $\mc{A}$ be a $k$-linear, Hom-finite additive category and assume that $\mc{A}$ and $\mc{A}^{\rm op}$ satisfy (IFP).
If $\mc{A}$ and $\mc{A}^{\rm op}$ satisfy (G),
then we have $\thick_{\mc{T}}\mc{M}=\U{\ms{GP}}(\ms{R}\mc{A},\mc{A})$.
\end{proposition}
\begin{proof}
Since $\mc{A}$ and $\mc{A}^{\rm op}$ satisfy (IFP), we have $\mc{M}\subset\U{\ms{GP}}(\ms{R}\mc{A},\mc{A})$.
Therefore we have $\thick\mc{M}:=\thick_{\mc{T}}\mc{M}\subset\U{\ms{GP}}(\ms{R}\mc{A},\mc{A})$.

Let $i\in\mb{Z}$ and $N\in\mod\mc{A}_{i}$.
Assume that $N$ has finite projective dimension over $\mc{A}_{i}$.
Since the inclusion $\mod\mc{A}_{i}\to\mod\ms{R}\mc{A}$ is exact, we have a resolution of $N$ by objects of the form $\mc{A}_{i}(-,X)$, $(X\in\mc{A})$ in $\mod\ms{R}\mc{A}$.
Therefore if $N$ is an object of $\ms{GP}(\ms{R}\mc{A},\mc{A})$, then $N$ is in $\thick\mc{M}$ if $\mc{A}_{i}(-,X)$ is in $\thick\mc{M}$ for any $X\in\mc{A}$.

Let $M\in\ms{GP}(\ms{R}\mc{A},\mc{A})$.
Since $M$ is a factor module of a finitely generated projective $\ms{R}\mc{A}$-module,
$\supp M$ is a finite set.
Thus by Lemma \ref{filtration-module} (b), $M$ has a finite filtration by $\rho_{i}(M)$ for $i=n,n+1,\ldots,m$, where $n=\min\supp M$ and $m=\max\supp M$.
By Lemma \ref{lem-rep-fun-GP} (b) and since $\rho (M)$ has finite projective dimension over $\mc{A}$, $\rho_{i}(M)\in\ms{GP}(\ms{R}\mc{A},\mc{A})$ for any $i\in\mb{Z}$.
Therefore $M$ is in $\thick\mc{M}$ if $\mc{A}_{i}(-,X)$ is in $\thick\mc{M}$ for any $X\in\mc{A}$ and $i=n,n+1,\ldots,m$.

We show that $\mc{A}_{i}(-,X)$ is in $\thick\mc{M}$ for any $X\in\mc{A}$ and $i\in\mb{Z}$ by an induction on $i$.
We first show $\mc{A}_{i}(-,X)\in\thick\mc{M}$ for $i\geq 0$.
Since $\mc{A}_{0}(-,X)\in\mc{M}$, we have $\mc{A}_{0}(-,X)\in\thick\mc{M}$.
Assume that $\mc{A}_{j}(-,X)\in\thick\mc{M}$ for $0\leq j\leq i-1$.
By Lemma \ref{embedding-composition}, we have an exact sequence in $\ms{GP}(\ms{R}\mc{A})$
\begin{align*}
0 \to \kD\mc{A}_{i-1}(X,-)\to\ms{R}\mc{A}(-,(X,i))\to\mc{A}_{i}(-,X)\to 0.
\end{align*}
Since $\kD\mc{A}_{i-1}(X,-)$ has finite projective dimension over $\mc{A}$ and by the inductive hypothesis, we have $\kD\mc{A}_{i-1}(X,-)\in\thick\mc{M}$.
Therefore $\mc{A}_{i}(-,X)$ is in $\thick\mc{M}$.

Next we show that $\mc{A}_{-i}(-,X)\in\thick\mc{M}$ for $i> 0$.
Assume that $\mc{A}_{-j}(-,X)\in\thick\mc{M}$ for $0\leq j\leq i-1$.
Let $n$ be the projective dimension of $\kD\mc{A}_{-i}(-,X)\simeq\kD(\mc{A}^{\rm op})_{i}(X,-)$ in $\mod(\mc{A}^{\rm op})_{i}$ and 
\begin{align*}
Q_{n} \xto{f} \cdots \to Q_{1}\to Q_{0} \to \kD\mc{A}_{-i}(-,X) \to 0
\end{align*}
be a minimal projective resolution in $\mod(\ms{R}\mc{A})^{\rm op}\simeq \mod\ms{R}(\mc{A}^{\rm op})$.
Put $K:=\Ker f$.
We have $K\in\ms{GP}(\ms{R}(\mc{A}^{\rm op}))$ by Lemmas \ref{lem-GP-dual-thick} (b) and \ref{lem-rep-fun-GP} (a).
By applying $\rho$ to this resolution, we have $K\in\ms{GP}(\ms{R}(\mc{A}^{\rm op}),\mc{A}^{\rm op})$.
Since the projective dimension of $\kD\mc{A}_{-i}(-,X)$ in $\mod(\mc{A}^{\rm op})_{i}$ is $n$ and by Lemma \ref{projective-cover} (c), we have $\rho_{i}(K)=0$.
Moreover by Lemma \ref{projective-cover} (c), we have $\rho_{\geq i+1}(K)=0$.
Therefore a $\ms{R}\mc{A}$-module $\kD K$ satisfies $\rho_{<-i+1}(\kD K)=0$.
Since $\kD K$ is a finitely generated $\ms{R}\mc{A}$-module, $\supp \kD K$ is finite.
Thus by Lemma \ref{filtration-module} (b), $\kD K$ has a finite filtration by $\rho_{j}(\kD K)$ for $-i+1\leq j \leq m$, where $m=\max\supp \kD K$.
By the inductive hypothesis, $\kD K\in\thick\mc{M}$ holds.
We have an exact sequence in $\ms{GP}(\ms{R}\mc{A})$
\begin{align*}
0\to \mc{A}_{-i}(-,X) \to \kD Q_{0} \to \kD Q_{1} \to \cdots \to \kD Q_{n} \to \kD K \to 0,
\end{align*}
where each $\kD Q_{l}$ is a projective $\ms{R}\mc{A}$-module.
This means $\mc{A}_{-i}(-,X)\simeq(\kD K)[-n-1]$ in $\U{\ms{GP}}(\ms{R}\mc{A},\mc{A})$.
Therefore we have $\mc{A}_{-i}(-,X)\in\thick\mc{M}$.
\end{proof}
\begin{proof}[Proof of Theorem \ref{thm-tilting-subcategory}]
 (a) This follows from Propositions \ref{condition-orthogonal} and \ref{condition-thick}.
 
 (b) Sine each object of $\mod\mc{A}$ has finite projective dimension, $\U{\ms{GP}}(\ms{R}\mc{A},\mc{A})=\U{\ms{GP}}(\ms{R}\mc{A})$ holds.
Thus the assertion follows from (a).
\end{proof}
\begin{proof}[Proof of Corollary \ref{cor-tilting-subcategory-dualizing}]
If $\mc{A}$ is a dualizing $k$-variety, then $\ms{GP}(\ms{R}\mc{A})=\mod\ms{R}\mc{A}$ holds.
The assertion directly follows from Theorem \ref{thm-tilting-subcategory}.
\end{proof}
\subsection{Happel's theorem for functor categories}\label{subsection-happel-functor}
As an application of Theorem \ref{thm-tilting-subcategory}, we show Happel's theorem for functor categories.
We need the following lemma.
\begin{lemma}\label{hom-of-tilting-subcategory}
Let $\mc{A}$ be a $k$-linear, Hom-finite additive category and assume that $\mc{A}$ and $\mc{A}^{\rm op}$ satisfy (IFP).
Let $X,Y\in\mc{A}$, $\mc{T}:=\U{\ms{GP}}(\ms{R}\mc{A})$.
We have the following equality:
\begin{align*}
	\mc{T}(\mc{A}_{0}(-,X),\mc{A}_{0}(-,Y)[n])\simeq\begin{cases}
					\mc{A}(X,Y) & n=0, \\
					0 & \textnormal{else}.
				\end{cases}
\end{align*}
\end{lemma}
\begin{proof}
By Proposition \ref{condition-orthogonal}, $\mc{T}(\mc{A}_{0}(-,X),\mc{A}_{0}(-,Y)[n\neq 0])=0$ holds.
Moreover we have
\begin{align}\label{A-to-RA-0}
(\Mod\ms{R}\mc{A})(\mc{A}_{0}(-,X),\ms{R}\mc{A}(-,(Y,0))) & \simeq (\Mod(\ms{R}\mc{A})^{\rm op})(\kD \ms{R}\mc{A}(-,(Y,0)),\kD\mc{A}_{0}(-,X))\nonumber\\
& \simeq (\Mod(\ms{R}\mc{A})^{\rm op})(\ms{R}\mc{A}((Y,-1),-),\kD\mc{A}_{0}(-,X))\nonumber\\
& \simeq \kD\mc{A}_{0}((Y,-1),X)=0,
\end{align}
where we use Lemma \ref{properties-of-repetitive} (b) and Yoneda's lemma.
By Lemma \ref{embedding-composition} (b), if a morphism $f : \mc{A}_{0}(-,X)\to\mc{A}_{0}(-,Y)$ in $\Mod\ms{R}\mc{A}$ factors though an object of $\proj\ms{R}\mc{A}$, then $f$ factors though $\ms{R}\mc{A}(-,(Y,0))$.
Thus by the equality (\ref{A-to-RA-0}), we have
\begin{align*}
\mc{T}(\mc{A}_{0}(-,X),\mc{A}_{0}(-,Y))=(\Mod\ms{R}\mc{A})(\mc{A}_{0}(-,X),\mc{A}_{0}(-,Y)).
\end{align*}
By applying the functor $(\Mod\ms{R}\mc{A})(-,\mc{A}_{0}(-,Y))$ to the exact sequence of Lemma \ref{embedding-composition} (b), since $(\Mod\ms{R}\mc{A})(\kD\mc{A}_{-1}(X,-),\mc{A}_{0}(-,Y))=0$ holds, we have
\begin{align*}
(\Mod\ms{R}\mc{A})(\mc{A}_{0}(-,X),\mc{A}_{0}(-,Y))&\simeq(\Mod\ms{R}\mc{A})(\ms{R}\mc{A}(-,(X,0)),\mc{A}_{0}(-,Y))\\
& \simeq \mc{A}_{0}((X,0),Y)\\
& \simeq \mc{A}(X,Y).
\end{align*}
\end{proof}
We have the following result, which is a functor category version of Happel's theorem.
\begin{corollary}\label{cor-happel-thm}
Let $\mc{A}$ be a $k$-linear, Hom-finite additive category and assume that $\mc{A}$ and $\mc{A}^{\rm op}$ satisfy (IFP).
\begin{itemize}
\item[{\rm (a)}]
If $\mc{A}$ and $\mc{A}^{\rm op}$ satisfy (G),
then we have a triangle equivalence \[\ms{K}^{\rm b}(\proj\mc{A})\simeq\U{\ms{GP}}(\ms{R}\mc{A},\mc{A}).\]

\item[{\rm (b)}]
If each object of $\mod\mc{A}$ and $\mod\mc{A}^{\rm op}$ has finite projective dimension,
then we have a triangle equivalence \[\ms{K}^{\rm b}(\proj\mc{A})\simeq\U{\ms{GP}}(\ms{R}\mc{A}).\]
\end{itemize}
\end{corollary}
\begin{proof}
(a)
Let $\mc{F}:=\ms{GP}(\ms{R}\mc{A},\mc{A})$ and $\mc{P}:=\proj\ms{R}\mc{A}$.
An inclusion functor $\proj\mc{A}\simeq\proj\mc{A}_{0}\to\mc{F}$ induces a triangle functor $\ms{K}^{\rm b}(\proj\mc{A})\to\ms{K}^{\rm-, b}(\mc{P})$.
Then we have the following triangle functors
\begin{align*}
F : \ms{K}^{\rm b}(\proj\mc{A}) \to \ms{K}^{\rm-, b}(\mc{P}) \to\ms{K}^{\rm-, b}(\mc{P})/\ms{K}^{\rm b}(\mc{P}) \to \U{\mc{F}},
\end{align*}
where the third is a quasi-inverse of Theorem \ref{rickard-theorem}.
We denote by $F$ the composite of these functors.
We show that $F$ is an equivalence by using Lemma \ref{basic-triangulated-functor}.

Put $\mc{U}:=\ms{K}^{\rm b}(\proj\mc{A})$ and $\mc{T}:=\U{\ms{GP}}(\ms{R}\mc{A},\mc{A})=\U{\mc{F}}$.
Note that $\proj\mc{A}$ is a subcategory of $\mc{U}$.
We show that a map
\begin{align*}
F_{M,N[n]} : \mc{U}(M,N) \to \mc{T}(FM,FN[n])
\end{align*}
is an isomorphism for any $M,N\in\proj\mc{A}$ and $n\in\mb{Z}$.
By Theorem \ref{rickard-theorem}, a quasi-inverse of $\ms{K}^{\rm-, b}(\mc{P})/\ms{K}^{\rm b}(\mc{P}) \to \U{\mc{F}}$ is induced from the composite of the canonical functors $\mc{F} \to \ms{K}^{\rm-, b}(\mc{P}) \to \ms{K}^{\rm-, b}(\mc{P})/\ms{K}^{\rm b}(\mc{P})$.
Therefore we have $F(\mc{A}(-,X))=\mc{A}_{0}(-,X)$ for any $X\in\mc{A}$.
For any $X,Y\in\mc{A}$, we have 
\begin{align*}
\mc{U}(\mc{A}(-,X),\mc{A}(-,Y))=\mc{A}(X,Y), \quad \mc{U}(\mc{A}(-,X),\mc{A}(-,Y)[n\neq 0])=0.
\end{align*}
Consequently, by Lemma \ref{hom-of-tilting-subcategory}, $F_{M,N[n]}$ is an isomorphism for any $M,N\in\proj\mc{A}$ and $n\in\mb{Z}$.

Since $\proj\mc{A}$ is Hom-finite and idempotent complete, so is $\ms{K}^{\rm b}(\proj\mc{A})$.
Clearly we have $\thick_{\mc{U}}(\proj\mc{A})=\mc{U}$.
Since $\Im(F|_{\proj\mc{A}})=\mc{M}$ holds, we have $\thick(\Im(F))=\mc{T}$ by Theorem \ref{thm-tilting-subcategory} (a).
Therefore $F$ is an equivalence by Lemma \ref{basic-triangulated-functor}.

(b)
Since each object of $\mod\mc{A}$ has finite projective dimension, we have $\U{\ms{GP}}(\ms{R}\mc{A},\mc{A})\simeq\U{\ms{GP}}(\ms{R}\mc{A})$.
Therefore we have the assertion by (a).
\end{proof}
\begin{corollary}\label{happel-thm-dualizing}
Let $\mc{A}$ be a dualizing $k$-variety.
If each object of $\mod\mc{A}$ and $\mod\mc{A}^{\rm op}$ has finite projective dimension,
then we have the following triangle equivalence
\begin{align*}
\ms{D}^{\rm b}(\mod\mc{A})\simeq\smod\ms{R}\mc{A}.
\end{align*}
\end{corollary}
\begin{proof}
If $\mc{A}$ is a dualizing $k$-variety, then $\ms{GP}(\ms{R}\mc{A})=\mod\ms{R}\mc{A}$ holds.
The assertion directly follows from Corollary \ref{cor-happel-thm}.
\end{proof}
\section{Proof of Theorem \ref{intro-thm-aim}}\label{section-application-hereditary}
Throughout this section, let $k$ be an algebraically closed field.
Let $A$ be a finite dimensional hereditary $k$-algebra, that is, $\gl(A)\leq 1$.
In this section, we apply Corollary \ref{happel-thm-dualizing} to $\smod A$ and show Theorem \ref{thm-hered-stable-derived}.

We denote by $\mod A$ the category of the finitely generated $A$-modules and denote by $\tau$ and $\tau^{-1}$ the Auslander-Reiten translations on $\mod A$.
We call an indecomposable $A$-module $M$ \emph{preprojective} (resp. \emph{preinjective}) if there exists an indecomposable projective $A$-module $P$ (resp. injective $A$-module $I$) and an integer $i$ such that $M\simeq \tau^{i}(P)$ (resp. $M\simeq\tau^{i}(I)$).
We call an indecomposable $A$-module $M$ \emph{regular} if $\tau^{i}(M)\neq 0$ for any $i\in\mb{Z}$.
Put the following subcategories of $\mod A$:
\begin{align*}
&\mc{P}:=\add\{ M\in\mod A \mid M \hspace{4pt} \textnormal{is a preprojective module}\},\\
&\mc{I}:=\add\{M\in\mod A\mid M \hspace{4pt} \textnormal{is a preinjective module}\},\\
&\mc{R}:=\add\{M\in\mod A\mid M \hspace{4pt} \textnormal{is a regular module}\}.
\end{align*}
We denote by $\ms{D}^{\rm b}(\mod A)$ the bounded derived category of $\mod A$ and denote by $\mb{S}$ a Serre functor of $\ms{D}^{\rm b}(\mod A)$. 
We regard $\mod A$ as a full subcategory of $\ms{D}^{\rm b}(\mod A)$ by the canonical inclusion.
Thus for any $X\in\ms{D}^{\rm b}(\mod A)$, $X\in\mod A$ if and only if $\Ho^{i}(X)=0$ for any $i\neq0$.

The following proposition is well known (see \cite[Chapter V\hspace{-.1em}I\hspace{-.1em}I\hspace{-.1em}I. 2.1. Proposition]{ASS} \cite[Chapter I, 5.2, Lemma]{H}).
\begin{proposition}\label{prop-trichotomy}
Let $A$ be a representation infinite hereditary algebra.
Then we have the following equalities.
\begin{align*}
&\ms{D}^{\rm b}(\mod A)=\bigvee_{i\in\mb{Z}}(\mod A)[i],\\
&\mod A=\mc{P}\vee\mc{R}\vee\mc{I}.
\end{align*}
\end{proposition}
We denote by $\mod_{\rm p}A$ the full subcategory of $\mod A$ consisting of modules without non-zero projective direct summands.
We define an additive functor \[\Phi : \ms{R}(\mod_{\rm p}A)\to \ms{D}^{\rm b}(\mod A)\] as follows.
For $X\in\mod_{\rm p}A$ and $i\in\mb{Z}$, let $\Phi(X,i):=\mb{S}^{i}(X)$.
For $X,Y\in\mod_{\rm p}A$ and $i,j\in\mb{Z}$, since $\mb{S}$ is a Serre functor of $\ms{D}^{\rm b}(\mod A)$, we have 
\begin{align*}
\Hom_{\ms{D}^{\rm b}(\mod A)}\bigl(\mb{S}^{i}(X), \mb{S}^{j}(Y)\bigr)\simeq
			\begin{cases}
				\Hom_{\ms{D}^{\rm b}(\mod A)}(X,Y) & i=j, \\
				\kD\Hom_{\ms{D}^{\rm b}(\mod A)}(Y,X) & j=i+1, \\
				0 & \textnormal{else},
			\end{cases}
\end{align*}
where the last isomorphism follows from Lemma \ref{SC0}.
By using these isomorphisms, we define a map
\begin{align*}
\Phi_{(X,i),(Y,j)}:\Hom_{\ms{R}(\mod_{\rm p}A)}((X,i),(Y,j))\to\Hom_{\ms{D}^{\rm b}(\mod A)}\bigl(\mb{S}^{i}(X), \mb{S}^{j}(Y)\bigr),
\end{align*}
and we extend $\Phi$ on $\ms{R}(\mod_{\rm p}A)$ additively.
$\Phi$ is actually a functor, since a Serre duality is bifunctorial.
\begin{lemma}\label{SC0}
Let $A$ be a representation infinite hereditary algebra.
For any $i<0$ and $j>1$, we have
\begin{align*}
\mb{S}^{i}(\mod_{\rm p}A)\subset \add(A)\vee\bigvee_{l < 0}\mod A\,[l],\quad
\mb{S}^{j}(\mod_{\rm p}A)\subset \add(\kD A)\vee\bigvee_{l > 1}\mod A\,[l].
\end{align*}
\end{lemma}
\begin{proof}
The assertions come from Proposition \ref{prop-trichotomy}.
\end{proof}
The first theorem of this section is the following.
Put $\mb{S}_{1}:=\mb{S}\circ[-1]$.
Note that $H^{0}(\mb{S}_{1}(M))\simeq \tau(M)$ and $H^{0}(\mb{S}^{-1}_{1}(M))\simeq \tau^{-1}(M)$ hold for any $M\in\mod A$.
\begin{theorem}\label{thm-rep-derived}
The functor $\Phi : \ms{R}(\mod_{\rm p}A)\to \ms{D}^{\rm b}(\mod A)$ is an equivalence of additive categories.
\end{theorem}
\begin{proof}
By the definition, $\Phi$ is fully faithful.
We show that $\Phi$ is dense.
Let $X$ be an indecomposable object of $\ms{D}^{\rm b}(\mod A)$.
By Proposition \ref{prop-trichotomy}, there exist an indecomposable $A$-module $M$ and an integer $l$ such that $X\simeq M[l]$.

Assume that $M$ is a preprojective module.
There exist an indecomposable projective $A$-module $P$ and $i\geq 0$ such that $M\simeq \mb{S}_{1}^{-i}(P)$.
If $i+l>0$, then we have $\mb{S}_{1}^{-(i+l)}(P)\in\mod_{\rm p}A$ and
\begin{align*}
\Phi(\mb{S}_{1}^{-(i+l)}(P),-l)&=\mb{S}^{l}(\mb{S}_{1}^{-(i+l)}(P))\\
&=\mb{S}_{1}^{-i}(P)[l].
\end{align*}
If $i+l\leq 0$, then we have $\mb{S}_{1}^{-(i+l)}(\mb{S}(P))\in\mod_{\rm p}A$ and
\begin{align*}
\Phi(\mb{S}_{1}^{-(i+l)}(\mb{S}(P)),-l+1)&=\mb{S}^{l-1}(\mb{S}_{1}^{-(i+l)}(\mb{S}(P)))\\
&=\mb{S}_{1}^{-i}(P)[l].
\end{align*}

Next assume that $M$ is a preinjective module.
There exist an indecomposable injective $A$-module $I$ and $i\geq 0$ such that $M\simeq \mb{S}_{1}^{i}(I)$.
If $i-l\geq 0$, then we have $\mb{S}_{1}^{i-l}(I)\in\mod_{\rm p}A$ and
\begin{align*}
\Phi(\mb{S}_{1}^{i-l}(I),-l)&=\mb{S}^{l}(\mb{S}_{1}^{i-l}(I))\\
&=\mb{S}_{1}^{i}(I)[l].
\end{align*}
If $i-l<0$, then we have $\mb{S}_{1}^{i-l}(\mb{S}^{-1}(I))\in\mod_{\rm p}A$ and
\begin{align*}
\Phi(\mb{S}_{1}^{i-l}(\mb{S}^{-1}(I)),-l-1)&=\mb{S}^{l+1}(\mb{S}_{1}^{i-l}(\mb{S}^{-1}(I)))\\
&=\mb{S}_{1}^{i}(I)[l].
\end{align*}

Assume that $M$ is a regular module.
Then we have $\mb{S}_{1}^{-l}(M)\in\mc{R}\subset\mod_{\rm p}A$ and $\Phi(\mb{S}_{1}^{-l}(M),-l)=\mb{S}^{l}(\mb{S}_{1}^{-l}(M))=M[l]$ holds.
Therefore the functor $\Phi : \ms{R}(\mod_{\rm p}A)\to \mc{D}$ is dense.
\end{proof}
Theorem \ref{thm-rep-derived} is an analog of the well known equivalence $\ms{D}^{\rm b}(\mc{H})\simeq\ms{Rep}\,\mc{H}$ for a hereditary abelian category $\mc{H}$ \cite[Theorem 3.1]{Lenzing}.
But they are quite different, since the definitions of $\ms{Rep}\,\mc{H}$ and $\ms{R}(\smod A)$ are quite different.

We recall the following proposition.
\begin{proposition}\cite[Propositions 6.2, 10.2]{AR74}\label{ar74-proposition-6.2-10.2}
Let $\mc{A}$ be a dualizing $k$-variety and $\mc{B}:=\mod\mc{A}$. Let $\mc{P}$ be the full subcategory of $\mc{B}$ consisting of the projective modules.
Then the following statements hold.
\begin{itemize}
\item[{\rm (a)}]
$\mc{B}/[\mc{P}]$ is a dualizing $k$-variety.
\item[{\rm (b)}]
Assume that the global dimension of $\mod\mc{A}$ is at most $n$,
then the global dimension of $\mod(\mc{B}/[\mc{P}])$ is at most $3n-1$.
\end{itemize}
\end{proposition}
Then we apply Corollary \ref{happel-thm-dualizing} to $\smod A$.
\begin{theorem}\label{thm-hered-stable-derived}
Let $A$ be a representation infinite hereditary algebra.
Then we have the following triangle equivalences
\begin{align*}
\smod\ms{D}^{\rm b}(\mod A)\simeq
\smod\ms{R}(\smod A)\simeq
\ms{D}^{\rm b}(\mod(\smod A)).
\end{align*}
\end{theorem}
\begin{proof}
Since $A$ is hereditary, a canonical functor $\mod_{\rm p}A\to\smod A$ induces an equivalence $\mod_{\rm p}A\simeq\smod A$.
Therefore the first equivalence comes from Theorem \ref{thm-rep-derived}.
By Proposition \ref{ar74-proposition-6.2-10.2}, $\smod A$ is a dualizing $k$-variety such that the global dimension of $\mod(\smod A)$ is at most two.
Therefore we can apply Corollary \ref{happel-thm-dualizing} to the dualizing $k$-variety $\smod A$.
We have the second equivalence.
\end{proof}
We say that two dualizing $k$-varieties $\mc{A}$ and $\mc{A}^{\prime}$ are derived equivalent if the derived categories of $\mod\mc{A}$ and $\mod\mc{A}^{\prime}$ are triangle equivalent.
\begin{corollary}\label{cor}
Let $A,A^{\prime}$ be representation infinite hereditary algebras.
If $A$ and $A^{\prime}$ are derived equivalent,
then $\smod A$ and $\smod A^{\prime}$ are derived equivalent.
\end{corollary}
\begin{remark}
If $A$ is a representation finite hereditary algebra, then Theorems \ref{thm-rep-derived}, \ref{thm-hered-stable-derived} and Corollary \ref{cor} were shown by \cite{IO}.
\end{remark}
\section*{Acknowledgements}
The author is supported by Grant-in-Aid for JSPS Fellowships 15J02465.
He would like to thank his supervisor Osamu Iyama for many supports and helpful comments.

\end{document}